\documentclass[a4paper,11pt]{amsart}

\usepackage{amsmath, amssymb, amsthm}
\usepackage{graphicx}
\usepackage[top=3cm, bottom=3cm, left=2.5cm, right=2.5cm]{geometry}
\usepackage[colorlinks=true, linkcolor=blue, citecolor = blue]{hyperref}
\usepackage{mathrsfs}
\usepackage{wasysym}
\usepackage{color}
\usepackage{bm}

\usepackage{enumitem}

\newtheorem{thm}{Theorem}[section]
\newtheorem{lem}[thm]{Lemma}

\newtheorem{prop}[thm]{Proposition}
\newtheorem{defn}[thm]{Definition}
\newtheorem{assumpt}{Assumption}[section]

\newtheorem{rmk}{Remark}[section]

\newcommand{\eps}{\varepsilon}

\newcommand{\Ord}{\mathcal{O}}

\newcommand{\om}{\omega}
\newcommand{\Om}{\Omega}

\newcommand{\expt}{\mathbb{E}}
\newcommand{\mbbP}{\mathbb{P}}

\newcommand{\R}{\mathbb{R}}

\newcommand{\C}{\mathcal{C}}
\newcommand{\Ch}{\hat{\C}}

\newcommand{\Th}{\hat{T}}
\newcommand{\Ind}{\mathbf{1}_{\{0\}}}
\newcommand{\la}{\langle}
\newcommand{\ra}{\rangle}
\newcommand{\pih}{\hat{\pi}}

\newcommand{\Psiz}{\tilde{\Psi}}

\newcommand{\pj}{\Pi}

\newcommand{\Xeps}{X^\eps}

\newcommand{\stopt}{\mathfrak{e}}
\newcommand{\stopteps}{\mathfrak{e}^\eps}

\newcommand{\stabsol}{\gamma}
\newcommand{\modcont}{\mathfrak{w}}

\newcommand{\pjeps}{{\pj}^{\eps}}
\newcommand{\hpex}[1]{{\pj}^{\eps}_{#1}\Xeps}
\newcommand{\zed}{\mathfrak{z}^\eps}
\newcommand{\zedh}{\widehat{\mathfrak{z}}^\eps}

\newcommand{\bfe}{\mathbf{e}}

\newcommand{\critsol}{\chi}
\newcommand{\mfXeps}{\mathfrak{X}^\eps}

\newcommand{\zpw}{\mathfrak{y}}
\newcommand{\mfw}{\mathfrak{w}}
\newcommand{\Tavg}{\mathbb{T}}

\newcommand{\stbp}{\mathscr{Y}}
\newcommand{\ydec}{\mathfrak{Y}}

\newcommand{\remG}{\mathfrak{g}}

\newcommand{\mfv}{\mathfrak{v}}

\newcommand{\zedt}{\pmb{\widetilde{\mathfrak{z}}}^\eps}\newcommand{\zedtz}{\pmb{\widetilde{\mathfrak{z}}}^0}

\newcommand{\B}{\mathcal{B}}

\newcommand{\hprc}{\mathcal{H}}

\title[Perturbations of critical DDE]{Approximation of delay differential equations at the verge of instability by equations without delay}
\author{Nishanth Lingala}
\address{University of Illinois\\
Urbana, IL, USA.}
\email{nlingala1@gmail.com}

\begin{document}

\begin{abstract}
We consider linear delay differential equations at the verge of Hopf instability, i.e. a pair of roots of the characteristic equation are on the imaginary axis of the complex plane and all other roots have negative real parts. When nonlinear and noise perturbations are present, we show that the error in approximating the dynamics of the delay system by certain two dimensional stochastic differential equation \emph{without delay} is small (in an appropriately defined sense). Two cases are considered: (i) linear perturbations and multiplicative noise (ii) cubic perturbations and additive noise. The two-dimensional system without-delay is related to the projection of the delay equation onto the space spanned by the eigenfunctions corresponding to the imaginary roots of the characteristic equation.

A part of this article is an attempt to relax the Lipschitz restriction imposed on the coefficients in \cite{LNNSN_SDDEavg} for additive noise case. Also, the multiplicative noise case is not considered in \cite{LNNSN_SDDEavg}. Examples without rigorous proofs are worked in \cite{LNPRLE}.
\end{abstract}

\maketitle

\section{Introduction}\label{sec:intro}

Consider the stochastic delay differential equation (SDDE)
\begin{align}\label{eq:motivsys_SDDE}
dx(t)=\left(\mu x(t-1)+x^3(t)\right)dt + \eps c_1x(t-1)dV_1(t) + \eps^2 c_2dV_2(t),
\end{align}
where $0<\eps\ll 1$ and $V_1,V_2$ are Wiener processes. The above equation represents a noisy perturbation of the following deterministic system:
\begin{align}\label{eq:motivsys}
\dot{x}(t)=\mu x(t-1)+x^3(t).
\end{align}
The linear system corresponding to \eqref{eq:motivsys} is \begin{align}\label{eq:motivsys_linear}
\dot{x}(t)=\mu x(t-1).
\end{align}
Seeking a solution of the form $x(t)=e^{t\lambda}$ to the linear system, we find that $\lambda$ must satisfy the characteristic equation $\lambda - \mu e^{-\lambda}=0$. When $\mu \in (-\frac{\pi}{2},0)$, all roots of the characteristic equation have negative real parts (see corollary 3.3 on page 53 of \cite{Stepanbook}). When $\mu=-\frac{\pi}{2}$ a pair of roots $\pm i\frac{\pi}{2}$ are on the imaginary axis and all other roots have negative real parts. When $\mu<-\frac{\pi}{2}$ some of the roots have positive real part.
Hence, the system \eqref{eq:motivsys_linear} is on the verge of instability at $\mu=-\frac{\pi}{2}$. Close to the verge of instability, the behaviour of the solution is oscillatory with amplitude increasing or decreasing depending on whether the root with the largest real part has positive real part or negative.

To study \eqref{eq:motivsys_SDDE} close to the verge of instability, set $\mu=-\frac{\pi}{2}+\eps^2\tilde{\mu}$ and zoom-in near $x=0$, i.e. write $y(t)=\eps^{-1}x(t)$ for $x$ governed by \eqref{eq:motivsys_SDDE}. We get 
\begin{align}\label{eq:exmpHpfyprc}
dy(t)=-\frac{\pi}{2}y(t-1)dt+\eps^2(y^3(t)+\tilde{\mu}y(t-1))dt+\eps c_1y(t-1)dV_1(t) + \eps c_2dV_2(t).
\end{align}
The equations studied in this paper are of the above kind. Before stating the equations in more precise terms below, we describe briefly the motivation for studying such equations.

Delay equations at the verge of instability arise, for example, in machining processes \cite{Gabor1}, in the response of eye-pupil to incident light \cite{Long90}, in human balancing \cite{Yao_delay_SDDE_stand}. In machining processes, the motion of the cutting tool can be described by a DDE---the tool cuts a work-piece placed on a rotating shaft and the delay is the time-period of the rotating shaft. For each rotation period there is certain rate of cutting below which the tool is stable and above which the tool breaks.  The inhomogenities in the properties of the material being cut can be modeled by noise (see \cite{buck_kusk_noise_machine_tool}). The eye-pupil exhibits oscillations in response to incident light---however there is some delay in the response because neurons have finite processing speed. This phenomenon can be modeled using a DDE at the verge of instability \cite{Long90}. So, a study of the effect of noise perturbations on `DDE at the verge of instability' is indeed useful.

Now we describe the equations studied in this article in more precise terms.

Let $\{y_t : t \geq 0\}$ be an $\R$-valued process governed by an SDDE. Let $r>0$ be the maximum of the delays involved in the drift and diffusion coefficients of the SDDE. To find the evolution at time $t$ of the process, we need to keep track of $y_s$  for $t-r\leq s \leq t$. For this purpose, let $\C:=C([-r,0];\R)$ be the space of $\R$-valued continuous functions on $[-r,0]$, and equip $\C$ with sup norm:
$$||\eta||:=\sup_{\theta \in [-r,0]}|\eta(\theta)|, \qquad \text{for } \eta\in \C.$$
Define the \emph{segment extractor} $\pj_t$ as follows: for $f\in C([-r,\infty);\R)$, $$(\pj_tf)(\theta)=f(t+\theta), \qquad \theta \in [-r,0], \quad t\in[0,\infty).$$
Then, consider equation of the form:
\begin{align}\label{eq:exmpHpfyprc_C1}
dy(t)=L_0(\pj_t y)dt+\eps^2G(\pj_ty)dt+\eps L_1(\pj_ty)dV_1(t) + \eps c_2dV_2(t), \qquad t\geq 0,
\end{align}
where $L_0,L_1:\C\to \R$ are bounded linear operators, $G:\C \to \R$, and $V_i$ are Wiener processes. Of course, as an initial condition we specify $\pj_0y=\xi \in \C$.

If we choose the maximum delay $r=1$ and $L_0(\eta)=-\frac{\pi}{2}\eta(-1)$, $L_1(\eta)=c_1\eta(-1)$, $G(\eta)=\eta^3(0)+\tilde{\mu}\eta(-1)$, we see that \eqref{eq:exmpHpfyprc_C1} represents \eqref{eq:exmpHpfyprc}.

We make the following assumption on $L_0$ to reflect the Hopf-bifurcation scenario:
\begin{assumpt}\label{ass:assumptondetsys}
We assume that the corresponding unperturbed DDE 
\begin{align}\label{eq:detDDE}
\dot{x}(t)=L_0(\pj_t x)
\end{align}
is critical, i.e. a pair of roots ($\pm i\om_c$) of the characteristic equation $\lambda-L_0(e^{\lambda \,\cdot})=0$ are on the imaginary axis (critical eigenvalues) and all other roots have negative real parts (stable eigenvalues). 
\end{assumpt}

Roughly speaking, under the above assumption, the solution of the unperturbed system \eqref{eq:detDDE} is oscillatory with constant amplitude. However for the perturbed system \eqref{eq:exmpHpfyprc_C1}, it can be shown that for the amplitude of oscillation of $y$ to change considerably, we need to wait for a time of order $\eps^{-2}$. Hence we change the time scale, i.e. define $Y^\eps(t)=y(t/\eps^2)$.

To be able to put the rescaled process $Y^\eps$ in a form akin to \eqref{eq:exmpHpfyprc_C1} we need to define the \emph{rescaled segment extractor} $\pjeps_t$ as follows:
for $f\in C([-\eps^2r,\infty);\R)$,  $$(\pjeps_tf)(\theta)=f(t+\eps^2\theta), \qquad \theta \in [-r,0], \quad t\in[0,\infty).$$
Then, \eqref{eq:exmpHpfyprc_C1}, with $Y^\eps(t)=y(t/\eps^2)$, can be written as
\begin{align}\label{eq:exmpHpfyprc_C}
dY^\eps(t)=\eps^{-2}L_0(\pjeps_t Y^\eps)dt+G(\pjeps_t Y^\eps)dt+ L_1(\pjeps_t Y^\eps)dW_1(t) + c_2dW_2(t),
\end{align}
where $W_i(t)=\eps V_i(t/\eps^2)$ are again Wiener processes.

In this paper, we restrict to equations of the form:
\begin{align}\label{eq:considerMAINadd}
\begin{cases}d\Xeps(t)&=\eps^{-2}L_0(\pjeps_t \Xeps)dt+G(\pjeps_t \Xeps)dt+ \sigma dW(t), \qquad t\in [0,T],\\
\Xeps(t)&=\xi(\eps^{-2}t), \qquad t\in[-\eps^2r,0], \quad \xi\in \C,\\
G(\eta)&=\int_{-r}^0\eta(\theta)d\nu_1(\theta)+\int_{-r}^0\eta^3(\theta)d\nu_3(\theta),
\end{cases}
\end{align}
where for $i=1,3$, $\nu_i:[-r,0]\to \R$, are bounded functions continuous from the left on $(-r,0)$ and normalized with $\nu_i(0)=0$; and also equations of the form:
\begin{align}\label{eq:considerMAINmul}
\begin{cases}d\Xeps(t)&=\eps^{-2}L_0(\pjeps_t \Xeps)dt+G(\pjeps_t \Xeps)dt+ L_1(\pjeps_t \Xeps) dW(t), \qquad t\in [0,T],\\
\Xeps(t)&=\xi(\eps^{-2}t), \qquad t\in[-\eps^2r,0], \quad \xi\in \C,\\
&|G(\eta_1)-G(\eta_2)|\,\leq\,K_G||\eta_1-\eta_2||, \qquad G(0)=0.
\end{cases}
\end{align}

We refer to \eqref{eq:considerMAINadd} as the \emph{additive noise case} and \eqref{eq:considerMAINmul} as the \emph{mulitplicative noise case}. In both cases we assume that the initial condition $\xi$ is deterministic (not a random variable).

Equations of the form $d\Xeps(t)=\eps^{-2}L_0(\pjeps_t \Xeps)dt+G(\pjeps_t \Xeps)dt+ \sigma dW(t)$ were studied in  \cite{LNNSN_SDDEavg} \emph{but the coefficient $G$ was assumed to be Lipschitz}. A quantity $\hprc^\eps$ was identified which, roughly speaking, gives the amplitude of oscillations of $\Xeps$. It was shown that the distribution of $\hprc^\eps$ converges weakly to the distribution of a process $\hprc^0$ governed by a stochastic differential equation (SDE) \emph{without delay}. For small $\eps$, this $\hprc^0$ gives good approximation for the dynamics of $\Xeps$. The advantage is three fold: (i) equations without delay are easier to analyze, (ii) for numerical simulations, $\Xeps$ requires storage of $\pjeps_t\Xeps$ (the entire segment) whereas $\hprc^0$ requires just the storage of current value $\hprc^0_t$, (iii) numerical simulation of $\Xeps$ requires very small time-step for integration because the drift coefficient is of the order $\eps^{-2}$, whereas $\hprc^0$ does not require  such a small time-step.

In this article we relax the Lipschitz assumption on the coefficient $G$ for the additive noise case. Note that the presence of $\nu_3$ in \eqref{eq:considerMAINadd} makes $G$ non-Lipschitz. The process $\hprc^\eps$ mentioned above encodes information only about the critical eigenspace (space spanned by the eigenvectors corresponding to the imaginary roots of the characteristic equation), and to obtain the convergence to $\hprc^0$ one needs to show that the projection of $\Xeps$ onto stable eigenspace is small (details would be provided later). In \cite{LNNSN_SDDEavg} this was easy to show because of the Lipschitz condition on $G$. In this article we need to follow a different approach.

The case of multiplicative noise \eqref{eq:considerMAINmul} is also considered here. However, for the multiplicative noise case the Lipschitz condition could not be relaxed. The presence of cubic nonlinearites causes the following problem: in trying to estimate a moment of certain order we encounter terms with higher order moments. 

\cite{LNPRLE} discusses the approaches in the literature towards SDDE at the verge of instability and shows the mistakes and shortcomings of those approaches (see section 1 and appendix A of \cite{LNPRLE}). Hence, here we refrain from mentioning these works again.

Though here we discuss rigorously only $\R$-valued processes, the multi-dimensional processes are dealt with in \cite{LNPRLE} without proofs. An applications-oriented reader would benefit from \cite{LNPRLE} rather than this article.

Before stating the goals of this paper, we give a brief overview of the unperturbed system \eqref{eq:detDDE}, and the variation-of-constants formula relating the solutions of \eqref{eq:considerMAINadd} and \eqref{eq:considerMAINmul} with \eqref{eq:detDDE}. The material in section \ref{subsec:unpertsys_overview} can be found in chapter 7 of \cite{Halebook} (see also \cite{Diekmanbook}).

\subsection{The unperturbed system \eqref{eq:detDDE}}\label{subsec:unpertsys_overview}

The solution of \eqref{eq:detDDE} gives rise to the strongly continuous semigroup $T(t):\C\to \C, \, t\geq 0$, defined by  
$T(t)\pj_0x=\pj_t x$.

The state space $\C$ can be split in the form $\C=P\oplus Q$ where $P=\text{span}_{\R}\{\Phi_1,\,\Phi_2\}$ where 
$$\Phi_1(\theta)=\cos(\om_c\theta), \qquad \Phi_2(\theta)=\sin(\om_c\theta), \qquad \theta\in[-r,0].$$
Write $\Phi=[\Phi_1,\,\Phi_2]$. Any $\eta \in P$ can be written as $\Phi z=z_1\Phi_1+z_2\Phi_2$ for $z\in \R^2$, i.e. $\Phi$ is a basis for the two-dimensional space $P$ and the $z$ are coordinates of $\eta \in P$ with respect to the basis $\Phi$.

Let $\pi$ denote the projection of $\C$ onto $P$ along $Q$, i.e. $\pi:\C\to P$ with $\pi^2=\pi$ and $\pi(\eta)=0$ for $\eta\in Q$. The operator $\pi$ can be written down explicitly, but we would not need the explicit form.

\subsubsection{Behaviour of the solution on $P$ and $Q$}

It is easy to see that $\pj_tx=\cos(\omega_c(t+\cdot))$ is a solution to \eqref{eq:detDDE} with the initial condition $\pj_0x=\cos(\omega_c\cdot)$, and  $\pj_tx=\sin(\omega_c(t+\cdot))$ is a solution to \eqref{eq:detDDE} with the initial condition $\pj_0x=\sin(\omega_c\cdot)$. Using the identity $\cos(\omega_c(t+\cdot))=\cos(\omega_ct)\cos(\omega_c\cdot)-\sin(\omega_ct)\sin(\omega_c\cdot)$ and the linearity of $L_0$, it can be shown that
\begin{align}\label{eq:TPhi_eq_PhieB}
T(t)\Phi(\cdot)=\Phi(\cdot)e^{Bt}, \qquad B=\left[\begin{array}{cc}0 & \omega_c \\ -\omega_c & 0 \end{array}\right].
\end{align}

There exists positive constants $\kappa$ and $K$ such that  
\begin{align}\label{eq:expdecesti}
||T(t)\phi||\leq Ke^{-\kappa t}||\phi||, \qquad \quad \forall\, \phi \in Q. 
\end{align}
The above is a consequence of the fact that, except for the roots $\pm i\om_c$ all other roots of the characteristic equation have negative real parts.

Write the solution to \eqref{eq:detDDE} as
$$\pj_tx=\pi\pj_tx+(1-\pi)\pj_tx=:\Phi z(t)+y_t$$ where $z$ is $\R^2$-valued and $y$ is $\C$ valued. Then we find that\footnote{Multiply \eqref{eq:TPhi_eq_PhieB} by $z(0)$ and realize (using the fact $T$ commutes with $\pi$) that $T(t)\Phi(\cdot)z(0)=T(t)\pi \pj_0x=\pi T(t)\pj_0x=\pi \pj_tx=\Phi z(t)$ to get that $\Phi z(t)=\Phi e^{Bt}z(0)$ from which $\dot{z}=Bz$ follows.} $z$ oscillate harmonically according to $\dot{z}(t)=Bz(t)$ and $||y_t||$ decays exponentially fast, i.e.
\begin{align}\label{eq:expdecunpertsys}
||y_t||\,\,\leq\,\,Ke^{-\kappa t}||y_0||.
\end{align}

\subsection{The variation-of-constants formula}
The solution of the perturbed systems \eqref{eq:considerMAINadd} or \eqref{eq:considerMAINmul} can be expressed in terms of the solution of \eqref{eq:detDDE} with the initial condition $\pj_0x=\Ind$ where 
$$\Ind(\theta)=\begin{cases} 1, \qquad \theta=0, \\
0, \qquad \theta \in [-r,0).\end{cases}$$
However, note that $\Ind$ does not belong to $\C$ and so we need to extend the space $\C$ to accommodate the discontinuity.

See p.192-193, 206-207 of \cite{SEAM_book} for the results pertaining to the extension. 
Let $\Ch:=\hat{C}([-r,0];\R)$ be the Banach space of all bounded measurable maps $[-r,0]\to \R$, given the sup norm. 
Solving the unperturbed system \eqref{eq:detDDE} for initial conditions in $\Ch$, we can extend the semigroup $T$ to one on $\Ch$. Denote the extension by $\Th$.
Again $\Ch$ splits in the form $\Ch=P\oplus \hat{Q}$. The projection $\pi$ can be extended to $\Ch$. The extension is denoted by $\pih$.  There exists a two component column vector $\Psiz \in \R^2$ such that
\begin{align}\label{eq:piIdeqPsiz}
\pih\Ind = \Phi \Psiz = \Psiz_1\Phi_1 + \Psiz_2\Phi_2. 
\end{align}
Also, there exists positive constants $\kappa$ and $K$ such that  
\begin{align}\label{eq:expdecesti}
||\Th(t)\phi||\leq Ke^{-\kappa t}||\phi||, \qquad \quad \forall\, \phi \in \hat{Q}. 
\end{align}

\subsubsection{Additive noise case} The solution of \eqref{eq:considerMAINadd} can be represented as (see theorem IV.4.1 on page 200 in \cite{SEAM_book})
\begin{align}\label{eq:vocformstateW_full_addn}
\hpex{t} \,\,=\,\,\Th(t/\eps^2)\hpex{0}\,\,&+\,\,\int_0^t\Th(\frac{t-u}{\eps^2})\Ind G(\hpex{u})du \,\,+\,\,\int_0^t\Th(\frac{t-u}{\eps^2})\Ind\, \sigma dW_u.
\end{align} 
The third term in the RHS of \eqref{eq:vocformstateW_full_addn} is an element in $\C$ and its value at $\theta \in [-r,0]$ is given by
$\int_0^t\left(\Th(\frac{t-u}{\eps^2})\Ind\right)(\theta)\, \sigma dW_u.$  Write $$\hpex{t}=\Phi z^\eps_t+y^\eps_t.$$ Here $(y^\eps_t)_{t\geq 0}$ is the $\C$-valued process $y^\eps_t=(1-\pi)\hpex{t}$ and $\Phi z_t = \pi\hpex{t}$. Note that $z$ is $\R^2$-valued process. Taking projection of \eqref{eq:vocformstateW_full_addn} onto the space $P$, and using the facts  (i) $\pi\hpex{t}=\Phi z^\eps_t$, (ii) $\Th$ commutes with $\pih$, (iii)  $\pih \Ind = \Phi \Psiz$, (iv) $\Th(t)\Phi z =\Phi e^{tB}z$, we get for $z^\eps$ (see corollary IV.4.1.1 on page 207 in \cite{SEAM_book})
\begin{align}\label{eq:zproj_addn}
dz^\eps_t=\eps^{-2}Bz^\eps_tdt+\Psiz G(\Phi z^\eps_t+y^\eps_t)dt+\Psiz \sigma dW_t, \qquad \Phi z^\eps_0=\pi \hpex{0}.
\end{align}
Using the fact that $\Th$ commutes with $\pih$, $y^\eps_t$ satisfies
\begin{align}\label{eq:vocformstateW_addn}
y^\eps_t \,\,=\,\,\Th(t/\eps^2)y^\eps_0\,\,&+\,\,\int_0^t\Th(\frac{t-u}{\eps^2})(1-\pih)\Ind G(\Phi z^\eps_u+y^\eps_u)du \,\,+\,\,\int_0^t\Th(\frac{t-u}{\eps^2})(1-\pih)\Ind\, \sigma dW_u.
\end{align}

\subsubsection{Multiplicative noise case} The solution of \eqref{eq:considerMAINmul} can be represented in a form analogous to \eqref{eq:vocformstateW_full_addn} with $\sigma$ replaced by $L_1(\hpex{u})$ (see \cite{Reis_Emery_Ineq_voc_for_SDDE}):
\begin{align}\label{eq:vocformstateW_full}
\hpex{t} \,\,=\,\,\Th(t/\eps^2)\hpex{0}\,\,&+\,\,\int_0^t\Th(\frac{t-u}{\eps^2})\Ind G(\hpex{u})du \,\,+\,\,\int_0^t\Th(\frac{t-u}{\eps^2})\Ind\, L_1(\hpex{u}) dW_u.
\end{align} 
For the projections onto $P$ and $Q$ we have:
\begin{align}\label{eq:zproj_muln}
dz^\eps_t=\eps^{-2}Bz^\eps_tdt+\Psiz G(\Phi z^\eps_t+y^\eps_t)dt+\Psiz L_1(\Phi z^\eps_t+y^\eps_t) dW, \qquad \Phi z^\eps_0=\pi \hpex{0},
\end{align}
\begin{align}\label{eq:vocformstateW}
y^\eps_t \,\,=\,\,\Th(t/\eps^2)y^\eps_0\,\,&+\,\,\int_0^t\Th(\frac{t-u}{\eps^2})(1-\pih)\Ind G(\Phi z^\eps_u+y^\eps_u)du\\ \notag  & \,\,+\,\,\int_0^t\Th(\frac{t-u}{\eps^2})(1-\pih)\Ind\, L_1(\Phi z^\eps_u+y^\eps_u) dW_u.
\end{align}

Crucial role would be played in proofs by
\begin{align}\label{eq:stabsoldefprob}
\stabsol(t):=(\Th(t)(1-\pih )\Ind)(0).
\end{align}
From \eqref{eq:expdecesti} we already know that 
\begin{align}\label{eq:stabsolexpdecesti}
|\stabsol(t)|\leq Ke^{-\kappa t}||(1-\pih )\Ind||, \qquad t\geq 0.
\end{align}
 Further, for $t> 0$ 
\begin{align}\label{eq:expdecayforhprime}
\left|\frac{d}{dt}\stabsol(t)\right|=\left|L_0\bigg(\Th(t)(1-\pih)\Ind\bigg)\right|\leq ||L_0||_{op}\,||\Th(t)(1-\pih)\Ind||\leq ||L_0||_{op}\,||(1-\pih)\Ind||\,Ke^{-\kappa t}.
\end{align}
Thus, both $\stabsol$ and $\stabsol'$ have exponential decay.

\subsection{Goal of this paper}\label{sec:subsec:goals}
Let $\Xeps$ evolve according to either \eqref{eq:considerMAINadd} or \eqref{eq:considerMAINmul}. Write $\hpex{t}=\Phi z^\eps_t+y^\eps_t,$ and define
\begin{align}\label{eq:ydecstbpdef}
\ydec^\eps_t:=\Th(t/\eps^2)y^\eps_0, \qquad \stbp^\eps_t:=y^\eps_t-\ydec^\eps_t.
\end{align}

Note that $\ydec^\eps_t$ depends only on the unperturbed system \eqref{eq:detDDE}. Given the initial condition $\pjeps_0 \Xeps$, $\ydec^\eps_t$ is a deterministic quantity. Note that $||\ydec^\eps_t||$ decays exponentially fast:
\begin{align}\label{eq:ydecexpfastdec}
||\ydec^\eps_t||\,\,\leq\,\,Ke^{-\kappa t/\eps^2}||(1-\pi)\pjeps_0 \Xeps||.
\end{align}

\subsubsection{Goal for the multiplicative noise case \eqref{eq:considerMAINmul}}\label{subsubsec:goalsformultipnoise}
Roughly speaking, the goals are
\begin{enumerate}[label=(\roman*)]
\item Show that, until time $T>0$, $\expt  \sup_{t\in[0,T]}||\stbp^{\eps}_t||^n\,\,\xrightarrow{\eps\to 0}0$, so that we can approximate  $y^\eps_t$ with the deterministic quantity $\ydec^\eps_t$
\item Consider the process 
\begin{align}\label{eq:zproj_muln_nodel}
d{\bf z}^\eps_t=\eps^{-2}B{\bf z}^\eps_tdt+\Psiz G(\Phi {\bf z}^\eps_t)dt+\Psiz L_1(\Phi {\bf z}^\eps_t) dW, \qquad \Phi {\bf z}^\eps_0=\pi \hpex{0}.
\end{align}
Note that ${\bf z}^\eps$ is two-dimensional process without delay totally ignoring $y^\eps$. 
Show that 
\begin{align}\label{eq:goalsformultipnoise_z}
\expt \sup_{t\in [0,T]}||{\bf z}^\eps_t-z^\eps_t||_2^2 \xrightarrow{\eps\to 0} 0
\end{align}
 where $||\cdot||_2$ is $\ell_2$ norm of vectors in $\R^2$.
\end{enumerate}

The above tasks justify the approximation of $\hpex{t}$ by $\Phi {\bf z}^\eps_t + \ydec^\eps_t$ for small $\eps$. Note that ${\bf z}^\eps$ is a two-dimensional process without delay and $\ydec^\eps_t$ is an exponentially decaying deterministic process. For small $\eps$ one could study this non-delay system instead of the original stochastic DDE \eqref{eq:considerMAINmul}. The advantage is that the 2-dimensional system without delay would be easier to analyze or simulate numerically. 

Further simplification can be obtained by studying the process 
\begin{align}\label{eq:hprcdefz2n}
\hprc^\eps_t:=\frac12||z^\eps_t||_2^2. 
\end{align}
Roughly speaking\footnote{Writing $\hpex{t}=\Phi z(t)+y_t$ and showing $y$ is small, we can write $\Xeps(t)=\hpex{t}(0)\approx \Phi(0)z_t = (z_t)_1$. Since the dynamics of $z$ is small perturbation of a predominant oscillation according $\dot{z}=Bz$, the approximate amplitude of $(z)_1$ is $\sqrt{(z)_1^2+(z)_2^2}=||z||_2=\sqrt{2\hprc}$.}, $\sqrt{2\hprc^\eps}$ \emph{is the amplitude of oscillations of} $\Xeps$. We will show that there is a constant $C$ such that 
\begin{align}\label{eq:hprcconvgreqzbound}
\expt \sup_{t\in [0,T]}||{\bf z}^\eps_t||_2^2<C
\end{align}
 for all $\eps$ smaller than some $\eps_*$. Using \eqref{eq:hprcconvgreqzbound} and \eqref{eq:goalsformultipnoise_z} it follows that 
\begin{align}\label{eq:hprcconvgmul}
\expt \sup_{t\in [0,T]}|\hprc^\eps_t-\frac12||{\bf z}^\eps_t||_2^2| \xrightarrow{\eps\to 0} 0.
\end{align}
 One can use standard averaging techniques for stochastic differential equations (without delay) to show that the distribution of $\frac12||{\bf z}^\eps||_2^2$ converges to the distribution of some one-dimensional process $\hprc^0$ without delay. 
By theorem 3.1 in \cite{Billingsley1999}, the distribution of $\hprc^\eps$ converges to the distribution of $\hprc^0$. 
For small $\eps$, the distribution of $\hprc^0$ gives a good approximation to the distribution of $\hprc^\eps$. The advantages of having $\hprc^0$ were mentioned in section \ref{sec:intro}.

\subsubsection{Goal for the additive noise case \eqref{eq:considerMAINadd}}\label{subsubsec:goal4addnoisecase}
The presence of cubic nonlinearites causes the following problem: in trying to estimate a moment of certain order we face the task of estimating terms with higher order moments. So the approach taken for \eqref{eq:considerMAINmul} does not work here. We take the following approach.

Recall the projection operator $\pi:\C \to P$. 
 Fix a constant $C_{\stopt}>0$ and define the stopping time $\stopteps=\inf\{t\geq 0:||\pi \pjeps_t\Xeps||\geq C_{\stopt}\}$. (Note that the stopping time depends on $\eps$). 
\begin{itemize}
\item Show that for $t\in [0,T\wedge\stopteps]$, $||\stbp^\eps_t||$ is small with high probability
\item Define a 2-dimensional process ${\bf z}^\eps$ as
\begin{align*}
d{\bf z}^\eps_t=\eps^{-2}B{\bf z}^\eps_tdt+\Psiz G(\Phi {\bf z}^\eps_t)dt+\Psiz \sigma dW, \qquad \Phi {\bf z}^\eps_0=\pi \hpex{0}.
\end{align*}
Note that ${\bf z}^\eps$ is a 2 dimensional process without delay.
Show that for $t\in [0,T\wedge\stopteps]$, error in approximating $z^\eps$ by ${\bf z}^\eps$ is small with high probability
\item Using estimates on ${\bf z}^\eps$ process, get rid of the stopping time and obtain approximation results until time $T$, by leveraging some arbitrarily small probability.
\end{itemize}

The stopping time helps in arriving at a bound on the norm of stable-mode $(1-\pi)\pjeps_t\Xeps$ without worrying about what happens to the critical-mode $\pi\pjeps_t\Xeps$.

\vspace{12pt}

Examples illustrating the usefulness of the above approximation results are shown in sections \ref{subsec:mul_examp} and \ref{sec:example}.

For related work on stochastic partial differential equations see \cite{Blom_amp_eq}. However note that in \cite{Blom_amp_eq} the bifurcation scenario is different---analogous situation in the DDE framework would be if one root of characteristic equation is zero and all other roots have negative real parts.

\section{Multiplicative noise}\label{sec:mulnoi}
In this section we consider \eqref{eq:considerMAINmul} with $T>0$ fixed. The constants here can depend on $T$.

The first goal is to show that $\expt \sup_{t\in [0,T]}||\stbp^{\eps}_t||^n\to 0$, which is the content of proposition \ref{prop:mul:stabnorm}. For this purpose, we use the variation of constants formulas \eqref{eq:vocformstateW_full}--\eqref{eq:vocformstateW}.
Recalling the definition of $\stbp^{\eps}$ from \eqref{eq:ydecstbpdef}, to estimate $\expt \sup_{t\in [0,T]}||\stbp^{\eps}_t||^n$, we need to estimate the last two terms on the RHS of \eqref{eq:vocformstateW}. 

Roughly speaking, the integral in the last term of RHS of \eqref{eq:vocformstateW} can be split as $\int_0^s=\int_0^{s-\eps^{\delta}r}+\int_{s-\eps^{\delta}r}^s$ with $0<\delta<2$. For $\int_0^{s-\eps^{\delta}r}$ we can use exponential decay of $\Th$ on $\hat{Q}$. For $\int_{s-\eps^{\delta}r}^s$, making note that the length of the interval of integration is small ($r\eps^\delta$), we need to be concerned with increments of Wiener process over small intervals, i.e. the modulus of continuity of the Wiener process.

Lemma \ref{lem:mul:quest:supuntilTisconst} is needed to be able to use the results of \cite{FischerNappo} on moments of modulus-of-continuity of Ito processes. Using results from \cite{FischerNappo}, proposition \ref{prop:mul:supUpsT} shows that the stochastic term in \eqref{eq:vocformstateW} is small. Then, straight forward estimation yields proposition \ref{prop:mul:stabnorm} which is the result that we need.

\begin{lem}\label{lem:mul:quest:supuntilTisconst}
Fix $n\geq 0$. There exists constants $\mathfrak{C}>0$ and $\eps_*>0$ such that $\forall \eps<\eps_*$,  
\begin{align}\label{eq:quest2}
\expt \sup_{t\in[0,T]}||\pjeps_t\Xeps||^n<\mathfrak{C}.
\end{align}
\end{lem}
Proof is given in appendix \ref{apsec:prf:lem:mul:quest:supuntilTisconst}. Note that, though one of the drift coefficients in \eqref{eq:considerMAINmul} is of the order $\eps^{-2}$, the constant $\mathfrak{C}$ above does not depend on $\eps$. Proof uses: (i) the variation-of-constants formula to exploit the exponential decay \eqref{eq:stabsolexpdecesti} and \eqref{eq:expdecayforhprime} on $\hat{Q}$, and oscillatory behaviour on $P$; (ii) Burkholder-Davis-Gundy inequality to estimate supremum of martingales by their quadratic-variation; and then (iii) Gronwall inequality.

\begin{defn}\label{def:modcont}
Define the modulus of continuity for $f:[0,\infty)\to \R$ by
$$\modcont(a,b;f)=\sup_{\substack{|u-v|\,\leq\, a \\ u,v\,\in\, [0,b]}}|f(u)-f(v)|.$$
\end{defn}

Define 
\begin{align}\label{eq:UpsandZdef}
\Upsilon^\eps_s:=\sup_{\theta \in [-r,0]}\left|\int_0^s\left(\Th(\frac{s-u}{\eps^2})(1-\pih)\Ind\right)(\theta) dZ_u\right|, \qquad Z_t:=\int_0^tL_1(\hpex{s}) d{W}_s.
\end{align}
Note that $Z$ dependens also on $\eps$.

\begin{prop}\label{prop:mul:supUpsT}
Fix $n\geq 1$. There exists constant $\hat{C}>0$ and a family of constants $\hat{\eps}_{\delta}>0$ (indexed by $0<\delta<2$) such that, given  $\delta \in (0,2)$ we have for $\eps < \hat{\eps}_{\delta}$  
\begin{align}\label{prop:mul:supUpsT_statement_mully}
\expt \sup_{s\in [0,T]}(\Upsilon^\eps_s)^n \,\,\,\leq\,\,\,\hat{C}\left({r\eps^\delta}\ln\left(\frac{2T}{r\eps^\delta}\right)\right)^{n/2} \,\,\xrightarrow{\eps\to 0} 0. 
\end{align} 
\end{prop}
Proof is given in appendix \ref{apsec:prf:prop:mul:supUpsT}. The essential idea of writing $\int_0^s=\int_0^{s-\eps^{\delta}r}+\int_{s-\eps^{\delta}r}^s$ and using \cite{FischerNappo} is mentioned earlier.


Let $\ydec^\eps_t$ and $\stbp^\eps_t$ be as defined in \eqref{eq:ydecstbpdef}.

\begin{prop}\label{prop:mul:stabnorm}
Fix $n\geq 1$. $\exists \eps_*>0$ such that $\forall\,\eps<\eps_*$,
\begin{align}\label{prop:mul:stabnorm_statement_b}
\expt  \sup_{s\in[0,T]}||\stbp^{\eps}_s||^n\,\,\leq\,\, \eps^{2n}2^{n-1}\left(\frac{K_GK}{\kappa}\right)^n\mathfrak{C}\,+\,2^{n-1}\expt\sup_{s\in [0,T]}(\Upsilon^{\eps}_s)^n\,\,\xrightarrow{\eps\to 0}0,
\end{align}
where $\mathfrak{C}$ is from lemma \ref{lem:mul:quest:supuntilTisconst}.
\end{prop}
Proof given in appendix \ref{apsec:prf:prop:mul:stabnorm}.

\vspace{16pt}

Recall that when we write $\hpex{t}=\Phi z^\eps_t+y^\eps_t$, the $\R^2$-valued process $z^\eps$ satisfies  equation \eqref{eq:zproj_muln}.
Removing the fast rotation induced by $B$, i.e. writing $\zed_t=e^{-tB/\eps^2}z^\eps_t$ we have
$$d\zed=e^{-tB/\eps^2}\Psiz G(\Phi e^{tB/\eps^2}\zed_t+y^\eps_t)dt+e^{-tB/\eps^2}\Psiz L_1(\Phi e^{tB/\eps^2} \zed_t+y^\eps_t) dW_t, \qquad \zed_0=z^\eps_0.$$
Let $\zedh$ be governed by
$$d\zedh=e^{-tB/\eps^2}\Psiz G(\Phi e^{tB/\eps^2}\zedh_t+\ydec^\eps_t)dt+e^{-tB/\eps^2}\Psiz L_1(\Phi e^{tB/\eps^2} \zedh_t+\ydec^\eps_t) dW_t, \qquad \zedh_0=\zed_0.$$
i.e. we are totally ignoring $y$ part except for the effect of initial condition (note that $\ydec^\eps_t=\Th(t/\eps^2)y^\eps_0$). 

As an intermediate step towards the end goal, we want to show that, until time $T$ the error in approximating $\zed$ by $\zedh$ is small. For this purpose, define 
$$\alpha^\eps_t=\frac12||\zed_t-\zedh_t||_2^2=\frac12((\zed_t-\zedh_t)_1^2+(\zed_t-\zedh_t)_2^2).$$
Here  $(\zed_t-\zedh_t)_i$ denotes the $i^{th}$ component of the $\R^2$-valued vector $\zed_t-\zedh_t$.
Let 
\begin{align}\label{eq:BigGammaDriftCoefDef}
\Gamma_t=\sum_{i=1}^2(\zed_t-\zedh_t)_i(e^{-tB/\eps^2}\Psiz)_i.
\end{align}
 Then $\alpha^\eps_t$ is governed by
$$d\alpha^\eps_t=\mathscr{B}_tdt+\Sigma_t dW_t, \qquad \alpha^\eps_0=0,$$
where
\begin{align*}
\mathscr{B}_t&=\Gamma_t\left(G(\Phi e^{tB/\eps^2}\zed_t+y^\eps_t)-G(\Phi e^{tB/\eps^2}\zedh_t+\ydec^\eps_t)\right)\\
&\qquad +\frac12 || \Psiz||_2^2\bigg(L_1(\Phi e^{tB/\eps^2} (\zed_t-\zedh_t))\,\,+\,\,L_1(y^\eps_t-\ydec^\eps_t)\bigg)^2,
\end{align*}
and
$$\Sigma_t = \Gamma_t\bigg(L_1(\Phi e^{tB/\eps^2} (\zed_t-\zedh_t))\,\,+\,\,L_1(y^\eps_t-\ydec^\eps_t)\bigg).$$

The following lemma gives processes dominating $\mathscr{B}_t$ and $\Sigma_t$. These help in applying Gronwall inequality to arrive at proposition \ref{prop:mul:alphaissmall}.
\begin{lem}\label{lem:mul:aux4comparison}
Define
\begin{align*}
\mathfrak{B}(\alpha,p)\,\,&:=\,\,C_{\mathscr{B}}(\alpha+p^2), \qquad C_{\mathscr{B}}=2|| \Psiz||_2^2||L_1||^2+3|| \Psiz||_2K_G,\\
\mathfrak{S}^2(\alpha,p)\,\,&:=\,\,C_{\Sigma}(\alpha^2+p^4), \qquad C_{\Sigma}=16|| \Psiz||_2^2||L_1||^2.
\end{align*}
Then $|\mathscr{B}_t|\leq \mathfrak{B}(\alpha^\eps_t,||\stbp^\eps_t||)$  and $\Sigma_t^2\leq\mathfrak{S}^2(\alpha^\eps_t,||\stbp^\eps_t||)$ for $t\geq0$.
\end{lem}
Proof given in appendix \ref{apsec:prf:lem:mul:aux4comparison}
\begin{prop}\label{prop:mul:alphaissmall}
Fix $\delta \in (0,2)$. There exists constants $C,\hat{\eps}_\delta>0$ such that $\forall \eps < \hat{\eps}_\delta$
\begin{align*}
\expt \sup_{s\in [0,T]}\left(\alpha^\eps_s\right)^2\,\,\leq&\,\,C\left(r\eps^\delta\ln(\frac{2T}{r\eps^\delta})\right)^2\,\,\xrightarrow{\eps\to 0} 0.
\end{align*}
\end{prop}
Proof is given in appendix \ref{apsec:prf:prop:mul:alphaissmall} and is by using lemma \ref{lem:mul:aux4comparison}, result \eqref{prop:mul:stabnorm_statement_b}, applying Gronwall pathwise (see \cite{StchGrwlScheu}) and Doob's $L^p$ inequalities.

As final step, consider the system
\begin{align}\label{eq:defzedtproc}
d\,\zedt_t=e^{-tB/\eps^2}\Psiz G(\Phi e^{tB/\eps^2}\zedt_t)dt+e^{-tB/\eps^2}\Psiz L_1(\Phi e^{tB/\eps^2} \zedt_t) dW, \qquad \zedt_0=z^\eps_0,
\end{align}
i.e. we are totally ignoring the $Q$ part---even the effect $\ydec$ of the initial condition. Define $\beta^\eps_t=\frac12||\zedt_t-\zedh_t||_2^2$. Using exactly the same technique as the one employed for $\alpha^\eps$ and using the exponential decay of $\ydec^\eps$, it is trivial to get the following result analogous to proposition \ref{prop:mul:alphaissmall}.
\begin{prop}\label{prop:mul:betaissmall}
There exists $C>0$ and $\eps_*>0$ such that $\forall \eps <\eps_*$ 
\begin{align*}
\expt \sup_{s\in [0,T]}\left(\beta^\eps_s\right)^2\,\,\leq&\,\,C\eps^2.
\end{align*}
\end{prop}
Proof given in appendix \ref{apsec:prf:prop:mul:betaissmall}.

Combining propositions \ref{prop:mul:stabnorm}, \ref{prop:mul:alphaissmall} and \ref{prop:mul:betaissmall} we get 
\begin{thm}\label{prop:mul:finalthem}
Fix $\delta \in (0,2)$. There exists constants $C,\hat{\eps}_\delta>0$ such that $\forall \eps < \hat{\eps}_\delta$
\begin{align}\label{eq:mul:finthem_zedhapprx}
\expt\sup_{t\in [0,T]}\big|X^\eps(t)-\big(\Phi(0)e^{tB/\eps^2}\zedh_t+\ydec^\eps_t(0)\big)\big|^4\,\,\leq\,\,C\left(r\eps^\delta\ln(\frac{2T}{r\eps^\delta})\right)^2\,\,\xrightarrow{\eps\to 0} 0.
\end{align}
There exists constants $C>0$ and $\eps_*>0$ such that $\forall \eps<\eps_*$
\begin{align}\label{eq:mul:finthem_zedtapprx}
\expt\sup_{t\in [0,T]}\big|X^\eps(t)-\big(\Phi(0)e^{tB/\eps^2}\zedt_t+\ydec^\eps_t(0)\big)\big|^4\,\,\leq\,\,C\eps^2.
\end{align}
\end{thm}
Proof given in appendix \ref{apsec:prf:prop:mul:mainTh}

Note that both the \emph{approximating processes} $\zedh$ and $\zedt$ are \emph{processes without delay}. However, $\zedh$ considers the effect of the initial condition $y^\eps_0$, but $\zedt$ ignores it. Hence the approximation \eqref{eq:mul:finthem_zedhapprx} using $\zedh$ is better than the approximation \eqref{eq:mul:finthem_zedtapprx}. For example, choosing $\delta$ close to two in \eqref{eq:mul:finthem_zedhapprx} we can get the bound $O(\eps^{4-})$ whereas the bound in \eqref{eq:mul:finthem_zedtapprx} is $O(\eps^2)$.

Now we revisit the goals stated in section \ref{subsubsec:goalsformultipnoise}. 

Note that for ${\bf z}^\eps_t$ defined in \eqref{eq:zproj_muln_nodel} we have ${\bf z}^\eps_t=e^{tB/\eps^2}\zedt_t$. Hence,
${\bf z}^\eps_t-z^\eps_t = e^{tB/\eps^2}(\zedt_t-\zed_t)$. Using the results of this section and the fact that for any $\R^2$-vector $v$, $||e^{tB/\eps^2}v||_2=||v||_2$, we can easily see that \eqref{eq:goalsformultipnoise_z} is satisfied.  The condition \eqref{eq:hprcconvgreqzbound} is equivalent to the following condition \eqref{eq:hprcconvgreqzbound_equiv}. Lemma \ref{lem:mul:auxzbound} is proved in appendix \ref{apsec:prf:lem:mul:auxzbound}.
\begin{lem}\label{lem:mul:auxzbound}
There exists constants $C$ and $\eps_*>0$ such that $\forall \eps<\eps_*$
\begin{align}\label{eq:hprcconvgreqzbound_equiv}
\expt \sup_{t\in [0,T]}||\zedt_t||_2^2\,\,<\,\,C.
\end{align}
\end{lem}
Hence, \eqref{eq:hprcconvgmul} follows. We summarize the discussion in section \ref{subsubsec:goalsformultipnoise} in the following theorem.
\begin{thm}
Define $\hprc^\eps_t:=\frac12||z^\eps_t||_2^2$ where $z^\eps$ are given by $\pi\hpex{t}=\Phi z^\eps_t$. Let $\zedt$ be the two-dimensional process (without delay) defined in \eqref{eq:defzedtproc}. Then 
\begin{align}\label{eq:hprcconvgmul_zedt}
\expt \sup_{t\in [0,T]}|\hprc^\eps_t-\frac12||\zedt_t||_2^2| \xrightarrow{\eps\to 0} 0.
\end{align}
If the process $\frac12||\zedt||_2^2$ converges weakly to a process $\hprc^0$, then $\hprc^\eps$ converges weakly to $\hprc^0$.
\end{thm}
\begin{rmk}
Because $\zedt$ is a process without delay, weak convergence of $\frac12||\zedt||_2^2$ can be dealt using standard averaging techniques for stochastic differential equations. 
\end{rmk}

\subsection{Example}\label{subsec:mul_examp}
Consider \eqref{eq:considerMAINmul} with $G\equiv 0$.
The corresponding $\zedt$ satisfies
\begin{align*}
d\zedt&=e^{-tB/\eps^2}M e^{tB/\eps^2} \zedt_t dW, \qquad M=\Psiz L_1\Phi,
\end{align*}
with $\zedt_0$ such that $\Phi \zedt_0=\pi\xi$.
Let $\hprc^\eps_t:=\frac12||z^\eps_t||_2^2$ and $H^\eps_t:=\frac12||\zedt_t||_2^2$.  
Then applying Ito formula we have
\begin{align}\label{eq:varsigeq}
dH^\eps_t=(\zedt_t)^*e^{-tB/\eps^2}M e^{tB/\eps^2} \zedt_t dW_t \,+\,\frac12\left(\left(e^{-tB/\eps^2}M e^{tB/\eps^2} \zedt_t\right)_1^2+\left(e^{-tB/\eps^2}M e^{tB/\eps^2} \zedt_t\right)_2^2\right)dt.
\end{align}
Averaging out the fast oscillations of $\zedt$, it can be shown that\footnote{ If $\frac12||\zedt||_2^2=\hprc^0$ then $$\lim_{\mathfrak{T}\to 0}\frac{1}{\mathfrak{T}}\int_0^{\mathfrak{T}}\frac12||e^{-tB/\eps^2}M e^{tB/\eps^2} \zedt||_2^2 \,dt\,=\, C_2\hprc^0  \quad \text{ and }\quad \lim_{\mathfrak{T}\to 0}\frac{1}{\mathfrak{T}}\int_0^{\mathfrak{T}}\left((\zedt)^*e^{-tB/\eps^2}M e^{tB/\eps^2} \zedt\right)^2dt \,=\, (C_1\hprc^0)^2.$$ } as $\eps \to 0$, the distribution of $H^\eps$ converges weakly to the distribution of
\begin{align}\label{eq:varsigeqavg}
d\hprc^0_t\,\,=\,\,C_{1}\hprc^0_t\, dW_t \,\,+\,\, C_2 \hprc^0_tdt, 
\end{align}
where $2C_1^2=3(M_{11}^2+M_{22}^2)+(M_{12}+M_{21})^2+ 2M_{11}M_{22}$ and $C_2=\frac12(\sum_{i,j=1}^2M_{ij}^2)$. Using $M=\Psiz L_1\Phi$ we get
\begin{align}\label{eq:varsigeqavg_fin}
d\hprc^0_t\,\,=\,\sqrt{\left(\frac12||\Psiz||_2^2\,||L_1\Phi||_2^2\,+\,(L_1\Phi\,\Psiz)^2\right)}\,\hprc^0_t\, dW_t \,\,+\,\, \frac12||\Psiz||_2^2\,||L_1\Phi||_2^2\, \hprc^0_tdt. 
\end{align}
For \eqref{eq:varsigeqavg_fin} solution can be written explicitly. For small $\eps$, the distribution of $\hprc^0$ gives good approximation to the distribution of $\hprc^\eps$. Note that, roughly speaking, $\sqrt{2\hprc^\eps}$ is the amplitude of oscillations of $\Xeps$. Hence $\hprc^0$ can be used to understand the dynamics of $\Xeps$. The advantage is that $\hprc^0$ does not involve any delay and is one-dimensional, and hence easier to analyze and simulate numerically (see \cite{LNPRLE} for examples involving numerical simulations).


\section{Additive noise}\label{sec:addnoi}
In this section we consider \eqref{eq:considerMAINadd}  with 
$T>0$ fixed. The constants here can depend on $T$.

The strategy employed for \eqref{eq:considerMAINmul} in the previous section does not work for \eqref{eq:considerMAINadd} due to the problem of moment-closure, i.e. in trying to estimate a lower moment we end up with the task of estimating a higher moment (because of the cubic nonlinearity). For \eqref{eq:considerMAINadd} we employ the strategy stated in section \ref{subsubsec:goal4addnoisecase}.

Define 
\begin{align}\label{eq:UpsandZdef_add}
\Upsilon^\eps_s:=\sup_{\theta \in [-r,0]}\left|\int_0^s\left(\Th(\frac{s-u}{\eps^2})(1-\pih)\Ind\right)(\theta) dZ_u\right|, \qquad Z_t:=\sigma W_t.
\end{align}

\begin{prop}\label{prop:add:supUpsT}
Fix $n\geq 1$. There exists constant $\hat{C}>0$ and a family of constants $\hat{\eps}_{\delta}>0$ (indexed by $0<\delta<2$) such that, given  $\delta \in (0,2)$ we have for $\eps < \hat{\eps}_{\delta}$  \\
\begin{align}\label{prop:add:supUpsT_statement_mully}
\expt \sup_{s\in [0,T]}(\Upsilon^\eps_s)^n \,\,\,\leq\,\,\,\hat{C}\left({r\eps^\delta}\ln\left(\frac{2T}{r\eps^\delta}\right)\right)^{n/2} \,\,\xrightarrow{\eps\to 0} 0. 
\end{align} 
\end{prop}
Proof is same as that of proposition \ref{prop:mul:supUpsT} with appropriate changes to account for $Z_t=\sigma W_t$; and we dont need anything analogous to lemma \ref{lem:mul:quest:supuntilTisconst}.

Fix a constant $C_{\stopt}>0$ and define the stopping time 
\begin{align}\label{eq:stoptimedefy}
\stopteps=\inf\{t\geq 0:||\pi \pjeps_t\Xeps||\geq C_{\stopt}\}.
\end{align}
The stopping time helps in arriving at a bound on the norm of stable-mode $(1-\pi)\pjeps_t\Xeps$ (until time $T\wedge \stopteps$) without worrying about what happens to the critical-mode $\pi\pjeps_t\Xeps$. Hence, as an intermediate step we establish results that hold until time $T\wedge \stopteps$ and later get rid of the stopping time $\stopteps$.

Write $\hpex{t}=\Phi z^\eps_t+y^\eps_t$. Then $z^\eps_t$ and $y^\eps_t$ satisfy the variation-of-constants formula \eqref{eq:zproj_addn} and \eqref{eq:vocformstateW_addn}. Define $\ydec^\eps$ and $\stbp^\eps$ as in \eqref{eq:ydecstbpdef}.

\begin{prop}\label{prop:add:stabnorm}
Let $\hat{C}$ and $\hat{\eps}_{\delta}$ be the same as in proposition \ref{prop:add:supUpsT}. There exists a family of constants $\eps_{a,C_{\stopt}}>0$ 
such that, given $a\in [0,1)$ and $\delta \in (2a,2)$, we have for $\eps < \min\{\hat{\eps}_{\delta},\,\eps_{a,C_{\stopt}}\}$
\begin{align}\label{add:newprop:supboundGron_statement_Gcub_lin_a_add}
\mbbP\bigg[\sup_{s\in [0,T\wedge\stopteps]}||\stbp^\eps_s|| \,\,\leq\,\,{{8\eps^a} } \bigg] \,\,\,\geq\,\,\,1-\hat{C}\eps^{-a}\sqrt{{r\eps^\delta}\ln\left(\frac{T}{r\eps^\delta}\right)}.
\end{align}
Here $\eps_{a,C_{\stopt}}$ is of the order $\Ord(\min\{C_{\stopt}^{-3/(2-a)},C_{\stopt}^{-3/2a}\})$ for large $C_{\stopt}$.
\end{prop}
In \eqref{add:newprop:supboundGron_statement_Gcub_lin_a_add} we obtain a bound on $||\stbp^\eps||$ which does not depend on $C_{\stopt}$ in spite of the cubic nonlinearity---hence the $\eps$ should be made really small. Larger the $C_{\stopt}$, smaller the $\eps$ we need to consider.
Proof is by straight forward application of exponential decay on $\hat{Q}$, Markov and Gronwall inequalities. Proof is given in appendix \ref{apsec:prf:prop:add:stabnorm}

\vspace{12pt}

Removing the fast rotation induced by $B$, i.e. writing $\zed_t=e^{-tB/\eps^2}z^\eps_t$ we have
$$d\zed_t=e^{-tB/\eps^2}\Psiz G(\Phi e^{tB/\eps^2}\zed_t+y^\eps_t)dt+e^{-tB/\eps^2}\Psiz \sigma dW_t, \qquad \zed_0=z^\eps_0.$$
Let $\zedh$ be governed by
$$d\zedh_t=e^{-tB/\eps^2}\Psiz G(\Phi e^{tB/\eps^2}\zedh_t+\ydec^\eps_t)dt+e^{-tB/\eps^2}\Psiz \sigma dW_t, \qquad \zedh_0=z^\eps_0,$$
i.e., in $\zedh$ we are totally ignoring $y$ part except for the effect of the initial condition ($\ydec^\eps_t=\Th(t/\eps^2)y^\eps_0$). Note that $\zedh$ is a process without delay.

We want to show that until time $T\wedge\stopteps$, error in approximating $\zed$ by $\zedh$ is small.  For this purpose, define $$\alpha^\eps_t=\frac12||\zed_t-\zedh_t||_2^2=\frac12((\zed_t-\zedh_t)_1^2+(\zed_t-\zedh_t)_2^2).$$
and let $\Gamma_t=\left(\sum_{i=1}^2(\zed_t-\zedh_t)_i(e^{-tB/\eps^2}\Psiz)_i\right)$. Then $\alpha^\eps_t$ is governed by
$$d\alpha^\eps_t\,\,=\,\,\mathscr{B}_t\,dt, \qquad \alpha^\eps_0=0,$$
where
\begin{align*}
\mathscr{B}_t&=\Gamma_t\left(G(\Phi e^{tB/\eps^2}\zed_t+y_t)-G(\Phi e^{tB/\eps^2}\zedh_t+\ydec^\eps_t)\right).
\end{align*}

The following lemma gives a process dominating $\mathscr{B}_t$. This helps in applying Gronwall inequality to arrive at proposition \ref{prop:add:alhaissmall}.

\begin{lem}\label{lem:add:aux4comparison}
$\exists\, C>0$ (is of the order $\Ord(C_{\stopt}^2)$ for large $C_{\stopt}$) such that if $\mathfrak{B}$ is defined by
\begin{align}\label{eq:def:mfBadd}
\mathfrak{B}(\alpha,p)\,\,&:=\,\,C\,\sqrt{2\alpha}\,\sum_{j=1}^3(p^j+(\sqrt{2\alpha})^j),
\end{align}
then $|\mathscr{B}_t|\leq \mathfrak{B}(\alpha_t,||\stbp^\eps_t||)$ for $t\in [0,T\wedge\stopteps]$.
\end{lem}
Proof given in appendix \ref{apsec:prf:lem:add:aux4comparison}.
\begin{prop}\label{prop:add:alhaissmall}
Let $\hat{C}$ and $\hat{\eps}_{\delta}$ be the same as in proposition \ref{prop:add:supUpsT}. There exists two families of constants $\eps_{a,C_{\stopt},1}>0$, $\eps_{a,C_{\stopt},2}>0$ 
such that, given $a\in [0,1)$ and $\delta \in (2a,2)$, we have for $\eps < \min\{\hat{\eps}_{\delta},\,\eps_{a,C_{\stopt},1},\,\eps_{a,C_{\stopt},2}\}$
\begin{align}\label{eq:pepsdefalp}
\mbbP[\sup_{t\in [0,T\wedge \stopteps]}\alpha^\eps_t \,\,\leq\,\,\eps^{a/2}]\,\,\geq \,\,1-\hat{C}\eps^{-a}\sqrt{{r\eps^\delta}\ln\left(\frac{T}{r\eps^\delta}\right)}\,\,=:\,\,p_{\eps} \xrightarrow{\eps \to 0}1. 
\end{align}
Here $\eps_{a,C_{\stopt},1}$ is of the order $\Ord(\min\{C_{\stopt}^{-3/(2-a)},C_{\stopt}^{-3/2a}\})$ (these are from proposition \ref{prop:add:stabnorm})
and $\eps_{a,C_{\stopt},2}$ is of the order $\Ord(\exp(-30C_{\stopt}^2T/a))$ for large $C_{\stopt}$.
\end{prop}
Proof is given in appendix \ref{apsec:prf:prop:add:alhaissmall}.
\vspace{12pt}

Finally, the stopping time $\stopteps$ can be got rid as follows.

Let $\Om$ be the set of all realizations of the Brownian motion $W$ and $\om \in \Om$ denote one particular realization.

\begin{defn}\label{def:propPTqDeff}
Given $T>0$ and $q>0$, we say that ``\emph{$\zedh$ system possesses the property $\mathscr{P}(T,q)$}'' if $\exists \,C_{\stopt},\, \eps_* \,>\,0$ such that $\forall \eps<\eps_*$, we have $\mbbP[E^\eps]\geq1-q$ where 
\begin{align}\label{assonzedh_setEeps}
E^\eps:=\left\{\om\,:\,\sup_{t\in [0,T]}||\Phi e^{tB/\eps^2}\zedh_t||\,<\,0.99C_{\stopt}\right\}.
\end{align} 
\end{defn}
\begin{thm}\label{prop:helpinprobconvg}
Fix  $T>0$.  Define
$$H^\eps:=\left\{\om\,:\,\sup_{t\in[0,T]}\alpha^{\eps}_t \,\geq\,\eps^{a/2}\right\}, \qquad S^\eps:=\left\{\om\,:\,\sup_{t\in[0,T]}||\stbp^\eps_t|| \,\geq\,8\eps^{a}\right\},$$
for $a\in [0,1)$.
Fix $q>0$ and assume $\zedh$ system possesses the property $\mathscr{P}(T,q)$. Then $\exists\,\eps_q>0$ such that $\forall\,\eps<\eps_q$,
\begin{align*}
\mbbP [H^\eps]\,\,<\,\,q+2(1-p_{\eps}), \qquad \mbbP [S^\eps]\,\,<\,\,q+2(1-p_{\eps}),
\end{align*}
where $p_{\eps}\to 1$ as $\eps\to 0$ and is given explicitly  in \eqref{eq:pepsdefalp}.
\end{thm}
Proof is given in appendix \ref{apsec:prf:prop:add:helpinprobconvg}.


Note that we have extended our results on $[0,T\wedge \stopteps]$ to $[0,T]$ by leveraging a small probability $q$, provided that $\zedh$ system possess property $\mathscr{P}(T,q)$.
Now we discuss under what conditions does $\zedh$ system possesses the property $\mathscr{P}(T,q)$ for arbitrary $q>0$.


Fix $T>0$. In general one cannot expect $\mathscr{P}(T,q)$ to hold for arbitrary $q>0$---for example, if the cubic nonlinearities have a destabilizing effect then there is a non-zero probability that trajectories blow-up in finite time. 
Similar situation arises in stochastic partial differential equations---see remark 5.2 in \cite{Blom_amp_eq}. When cubic nonlinearities have stabilizing effect, it is reasonable to expect $\mathscr{P}(T,q)$ to hold for arbitrary $q>0$ (see proposition \ref{lem:add:propertyPTq_stable} below).

The following two propositions help in checking if the property $\mathscr{P}(T,q)$ is satisfied. Proofs of them are similar in nature to the proof of Theorem 5.1 in \cite{Blom_amp_eq}.  \cite{Blom_amp_eq} deals with stochastic partial differential equations and the instability scenario there is different---analogous situation in delay equations case would be that ``one root of the characteristic equation is zero, and all other roots have negative real parts''. For the scenario that we are considering in this paper, one pair of roots lie on the imaginary axis, and so there are oscillations in the system and the proofs requires a bit more work than that in \cite{Blom_amp_eq}.

Proposition \ref{lem:add:propertyPTq} does not assume anything about the nature of the nonlinearity $G$---consequently its result is weak. Proposition \ref{lem:add:propertyPTq_stable} assumes that the nonlinearity is stabilizing and concludes that $\zedh$ possesses the property $\mathscr{P}(T,q)$ for any $q>0$.

\begin{prop}\label{lem:add:propertyPTq}
Fix $q>0$. Then $\exists T_q>0$ such that the $\zedh$ system  possesses the property $\mathscr{P}(T,q)$ for $T\in [0,T_q]$.
\end{prop}
Proof is given in appendix \ref{appsec:lem:add:propertyPTq}.

\begin{prop}\label{lem:add:propertyPTq_stable}
Fix $T>0$. Assume the cubic nonlinearity of $G$ is stabilizing, i.e., $\exists C_G>0$ such that
\begin{align}\label{lem:add:propertyPTq_stable_condition}
\frac{\om_c}{2\pi}\int_0^{2\pi/\om_c}\left((e^{tB}z)^*\Psiz\int_{-r}^0(\Phi(\theta)e^{tB}z)^3d\nu_3(\theta)\right)dt \,\,<\,\,-C_G||z||_2^4, \qquad \forall z\in \R^2.
\end{align}
Then the $\zedh$ system possesses the property $\mathscr{P}(T,q)$ for arbitrary $q>0$.
\end{prop}
Proof is given in appendix \ref{appsec:lem:add:propertyPTq_stableNony}.

\vspace{12pt}

Now consider the system
$$d\,\zedt_t=e^{-tB/\eps^2}\Psiz G(\Phi e^{tB/\eps^2}\zedt_t)dt+e^{-tB/\eps^2}\Psiz L_1(\Phi e^{tB/\eps^2} \zedt_t) dW, \qquad \zedt_0=z^\eps_0,$$
i.e. we are totally ignoring the $Q$ part---even the effect $\ydec$ of the initial condition. Define $$\beta^\eps_t=\frac12||\zedt_t-\zedh_t||_2^2.$$
\begin{prop}\label{prop:betaepsSmallAddNoiseFin}
Assume the cubic nonlinearity is such that \eqref{lem:add:propertyPTq_stable_condition} is satisfied, i.e. nonlienarity is stabilizing. Fix $T>0$. Given any $q>0$, $\exists\,C>0$ and $\eps_{\circ}>0$ such that $\forall \eps<\eps_{\circ}$
$$\mbbP\left[\sup_{t\in [0,T]}\beta^\eps_t \geq C\eps^4\right]\leq q.$$
\end{prop}
Proof is in appendix \ref{sec:proofofprop_betaepsSmallAddNoiseFin}.

Comibining theorem \ref{prop:helpinprobconvg} and proposition \ref{prop:betaepsSmallAddNoiseFin} we get the following result.
\begin{thm}\label{thm:cubnonstab_probconvg_apprx}
Assume the cubic nonlinearity is such that \eqref{lem:add:propertyPTq_stable_condition} is satisfied, i.e. nonlinearity is stabilizing. Fix any $a\in [0,1)$. For any $q>0$, $\exists \eps_q>0$ such that $\forall \eps <\eps_q$
\begin{align*}
\mbbP\bigg[\sup_{s\in [0,T]}\big|\Xeps_s- &\left(\Phi(0)e^{sB/\eps^2}\zedt_s+\ydec^\eps_s\right)\big|> 6\eps^{a/4}\bigg] \,< \,  3q + 4(1-p_\eps)
\end{align*}
where $p_{\eps}\to 1$ as $\eps\to 0$ and is given explicitly  in \eqref{eq:pepsdefalp}.
\end{thm}
\begin{proof}
Using $\Xeps_s = \Phi(0)e^{sB/\eps^2}\zed_s+\stbp^\eps_s + \ydec^\eps_s$ the above probability is bounded by
\begin{align*}
\mbbP\bigg[\sup_{s\in [0,T]}& \big|\Phi(0)e^{sB/\eps^2}(\zed_s-\zedh_s)\big|  +\big|\Phi(0)e^{sB/\eps^2}(\zedh_s-\zedt_s)\big|+|\stbp^\eps_s(0)|>6\eps^{a/4}\bigg] \\
&\leq \,\mbbP\left[\sup_{s\in [0,T]}\sqrt{2\alpha^\eps_t}>2\eps^{a/4}\right]+\mbbP\left[\sup_{s\in [0,T]}\sqrt{2\beta^\eps_t}>2\eps^{a/4}\right]+\mbbP\left[\sup_{s\in [0,T]}||\stbp^\eps_s||>2\eps^{a/4}\right] \\
& < q+2(1-p_{\eps}) \,\,+\,\,q\,\,+\,\, q+2(1-p_{\eps}) .\qquad \quad (\text{for } \eps \text{ sufficiently small.})
\end{align*}
\end{proof}

\begin{rmk}
Note that $\zedt$ is a 2-dimensional system without delay and $\ydec^\eps$ is a deterministic process that has exponential decay. The above theorem shows that, for small enough $\eps$, the delay system $\Xeps$ can be approximated by the $\zedt$ system without delay, with probability close to 1.
\end{rmk}

When the nonlinearity is stabilizing, using standard averaging techniques for equations without delay  (see for example \cite{NavamSowers}), it can be shown that the distribution of $\zedt$ converges as $\eps\to 0$ to the distribution of a 2-dimensional process $\zedtz$. Theorem \ref{prop:helpinprobconvg} and  propositions \ref{prop:betaepsSmallAddNoiseFin} show that $\sup_{t\in [0,T]}\beta^\eps_t$ and $\sup_{t\in [0,T]}\alpha^\eps_t$ converge to zero in probability. Hence, by theorem 3.1 in \cite{Billingsley1999}, the distribution of $\zed$ converges as $\eps \to 0$ to the distribution of $\zedtz$. Also, the distribution of $\hprc^\eps$ process, where $\hprc^\eps_t:=\frac12||z^\eps_t||_2^2=\frac12||\zed_t||_2^2$, converges as $\eps \to 0$ to the distribution of $\hprc^0$, where $\hprc^0$ is the weak-limit as $\eps \to 0$ of the process $H^\eps_t=\frac12||\zedt_t||_2^2$.

Great simplification can be obtained when $\frac12||\zed||_2^2$ can be used to approximate the required quantities. For example, consider the exit time
\begin{align}
\tau^\eps&:=\inf \{t\geq 0: |\Xeps_t|\geq \sqrt{2H^*}\}
\end{align}
where $H^*$ is fixed and is such that $\sqrt{2H^*}\gg ||(I-\pi)\pjeps_0 \Xeps||$. Noting that $\Phi(0)e^{tB/\eps^2}\zed_{t}=(\zed_t)_1\cos(\om t/\eps^2)+(\zed_t)_2\sin(\om t/\eps^2)$; because of the fast oscillations of $\zed$ and fast decay of $\ydec^\eps$ and smallness of $\stbp^\eps$, the exit time $\tau^\eps$ would be very close to the exit time ${\tau'}^\eps$ where
\begin{align}
{\tau'}^\eps&:=\inf \{t\geq 0: \sqrt{(\zed_t)_1^2+(\zed_t)_2^2}\geq \sqrt{2H^*}\}.
\end{align}
To approximate the distribution of $\tau^\eps$, one can study $H^\eps_t:=\frac12||\zedt_t||_2^2$ and consider the distribution of 
\begin{align*}
\tau^{\eps,\hbar}&:=\inf \{t\geq 0: H^\eps_t\geq H^*\}.
\end{align*}
The distribution of $\tau^{\eps,\hbar}$ would be close to that of $\tau^\eps$. Since $\zedt_t$ does not involve any delay, standard averaging techniques can be used to show that  the distribution of $H^\eps$ converges as $\eps\to 0$ to the distribution of a specific 1-dimensional process $\hprc^0$ (without delay). Then the exit times
\begin{align*}
\tau^{\hbar}&:=\inf \{t\geq 0: \hprc^0_t\geq H^*\}
\end{align*}
would closely approximate $\tau^\eps$. The advantages in doing so are: (i) $\hprc^0$ is a process without delay and hence easier to simulate (ii) numerical simulation of $\hprc^0$ can be done with a much coarser numerical mesh than that required for $\Xeps$.

\subsection{Example}\label{sec:example}

Consider the following equation:
\begin{align}\label{eq:examplesys_quadadded}
d\Xeps(t)=\eps^{-2}L_0(\pjeps_t \Xeps)dt + G(\pjeps_t\Xeps)dt + \sigma dW, \qquad \quad \pjeps_0\Xeps=\xi,
\end{align}
where $L_0\eta=-\frac{\pi}{2}\eta(-1)$ and $G(\eta)=\gamma_c\eta^3(-1)$. The characteristic equation $\lambda+\frac{\pi}{2}e^{-\lambda}=0$ has countably infinite roots on the complex plane. The roots with the largest real part are $\pm i\frac{\pi}{2}$. Hence $L_0$ satisfies the assumption \ref{ass:assumptondetsys}. The basis $\Phi$ for $P$ and the vector $\Psiz$ can be evaluated as
$$\Phi(\theta)=[\cos(\frac{\pi}{2}\theta)\,\,\,\sin(\frac{\pi}{2}\theta)], \qquad \quad \Psiz=\frac{2}{(1+(\pi/2)^2)}\left[\begin{array}{c}1 \\ \pi/2 \end{array}\right].$$
The corresponding $\zedt$ satisfies
\begin{align*}
d\zedt \,=\,\gamma_c e^{-tB/\eps^2}\Psiz\left(\Phi(-1)e^{tB/\eps^2}\zedt_t\right)^3dt +    e^{-tB/\eps^2}\Psiz\sigma dW, \qquad \quad \Phi \zedt_0=\pi\xi.
\end{align*}
Let $H^\eps_t=\frac12((\zedt_1)^2+(\zedt_2)^2)_t$. Applying Ito formula we have
\begin{align*}
dH^\eps_t\,&=\,\gamma_c(e^{-tB/\eps^2}\Psiz)^*\,\zedt (\Phi(-1)e^{tB/\eps^2}\zedt)^3\,dt + (e^{-tB/\eps^2}\Psiz)^*\,\zedt \sigma dW \\
& \qquad \qquad + \frac12\sigma^2\left((e^{-tB/\eps^2}\Psiz)_1^2+(e^{-tB/\eps^2}\Psiz)_2^2\right)dt.
\end{align*}
Averaging the fast oscillations we get that the probability distribution of $H^\eps$ converges as $\eps \to 0$ to the probability distribution of
\begin{align}\label{eq:examplesysavgd_quadadded}
d\hprc^0_t=\left(-\frac{3\gamma_c\pi/2}{1+(\pi/2)^2}(\hprc^0_t)^2 + \frac{2\sigma^2}{1+(\pi/2)^2}\right)dt + \frac{2}{\sqrt{1+(\pi/2)^2}}\sqrt{\hprc^0_t}\,\sigma dW.
\end{align}

Now we illustrate our results employing numerical simulations.

Draw a random sample of $N_{samp}$ particles with initial $H$ values $\{h_i\}_{i=1}^{Nsamp}$. Simulate them according to  \eqref{eq:examplesysavgd_quadadded} for $0\leq t \leq T_{end}$.

Simulate \eqref{eq:examplesys_quadadded} for $0 \leq t \leq T_{end}$ using initial trajectories $\{\sqrt{2h_i} \cos(\om_c\cdot)\}_{i=1}^{Nsamp}$. 

Let $\tau^\eps:=\inf \{t\geq 0: |\Xeps_t|\geq \sqrt{2H^*}\}$ and $\tau^h:=\inf \{t\geq 0: H^0(t)\geq H^*\}$

We can check whether the following pairs are close.
\begin{enumerate}
\item the distribution of $\frac12((z^\eps_{T_{end}})_1^2+(z^\eps_{T_{end}})_2^2)$ from \eqref{eq:examplesys_quadadded} \emph{and} the distribution of $\hprc^0_{T_{end}}$ from \eqref{eq:examplesysavgd_quadadded},
\item distribution of $\tau^\eps$  \emph{and} the distribution of $\tau^h$.
\end{enumerate}

We took $\eps=0.025$,  $H^*=1.5$, $T_{end}=2$, $N_{samp}=4000$, and $\sqrt{2\{h_i\}_{i=1}^{Nsamp}}=1.2$. Figures \ref{fig:tcdf} and \ref{fig:taucdf} answer the above questions. Two cases are considered with $\sigma=1$ fixed:  $\gamma_c=1$ and $\gamma_c=0$.

\begin{figure}
\centering
\begin{minipage}{.5\textwidth}
  \centering
\includegraphics[scale=0.45]{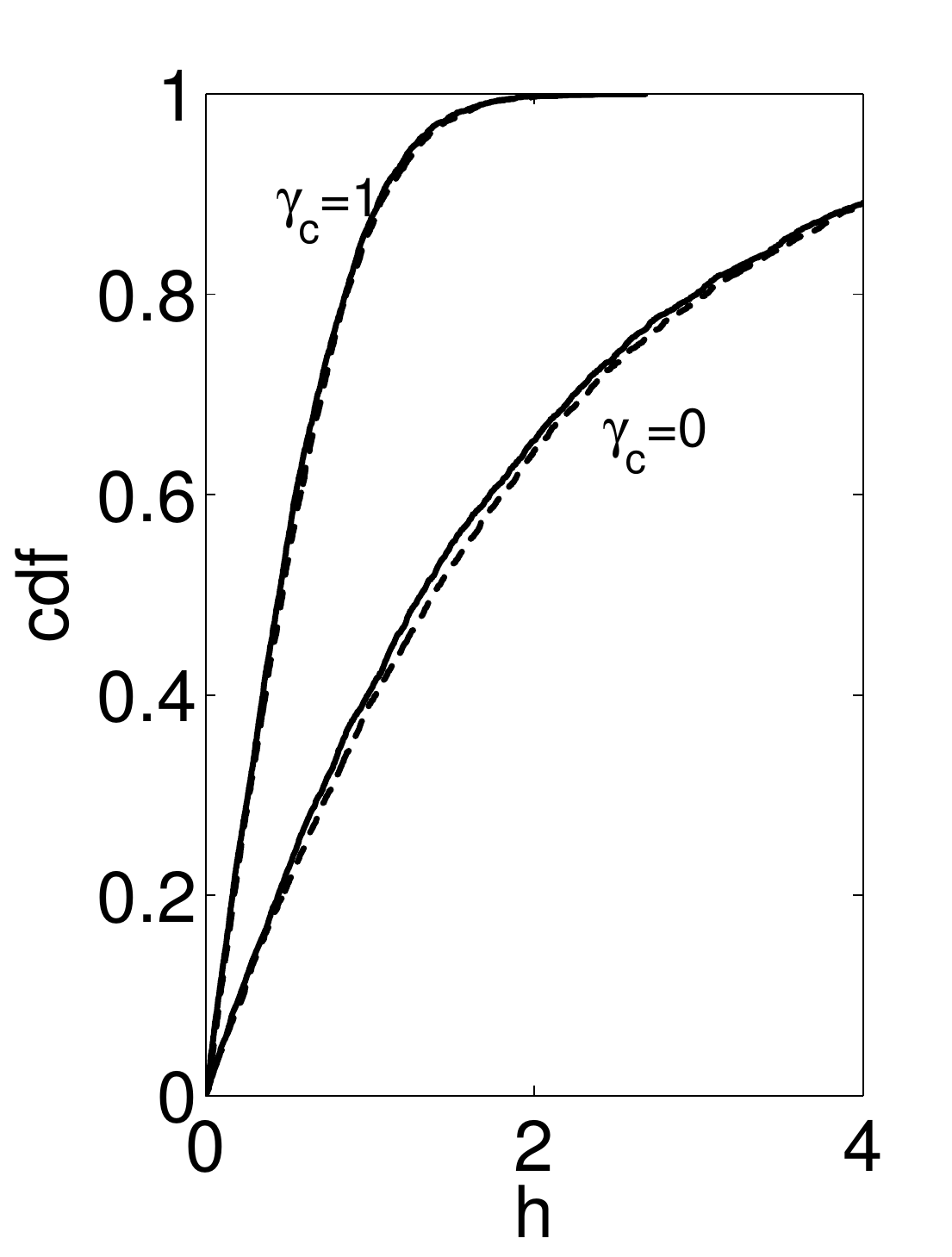}
  \caption{cdf of $\frac12((z^\eps_T)_1^2+(z^\eps_T)_2^2)$ (dashed line) and $\hprc^0_T$ (solid line).}
  \label{fig:tcdf}
\end{minipage}%
\begin{minipage}{.5\textwidth}
  \centering
  \includegraphics[scale=0.4]{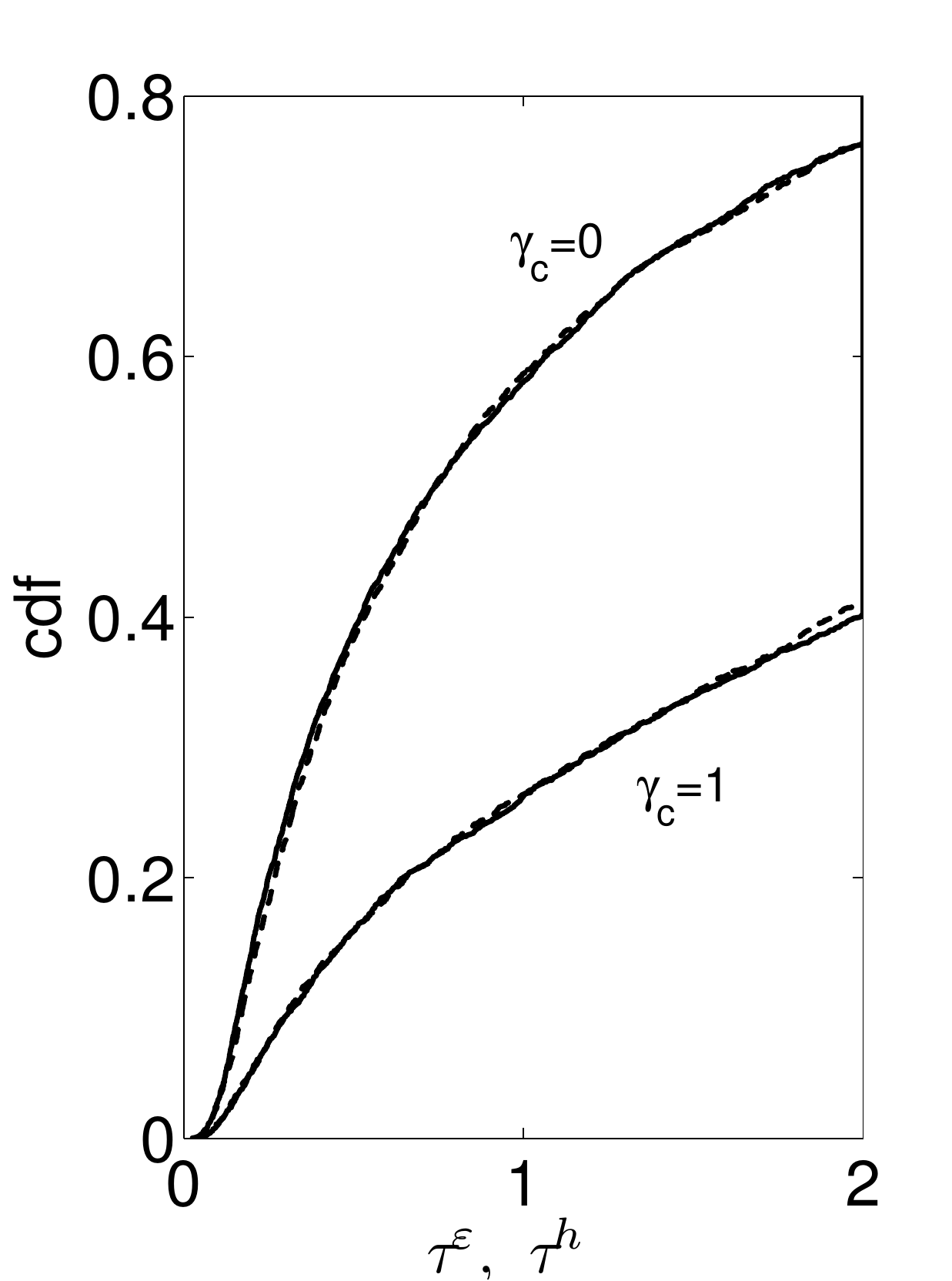}
  \caption{cdf of $\tau^\eps$ (dashed line) and cdf of $\tau^h$  (solid line). The cdf value at $\tau^\eps=2$ indicates the fraction of particles whose modulus exceeded $\sqrt{2H^*}$ before the time $T=2$.}
  \label{fig:taucdf}
\end{minipage}
\end{figure}

More examples (oscillators with cubic nonlinearity) are discussed in\footnote{\cite{LNPRLE} employs complex coordinates and so the form of answers would differ from this paper. However the numerical values would be same. For example, in this paper we write an element in $P$ as $z_1\cos(\om_c\cdot)+z_2\sin(\om_c\cdot)$ with $z_i \in \R$. But in \cite{LNPRLE} we write $z_1e^{i\om_c\cdot}+z_2e^{-i\om_c\cdot}$ with $z_1$ and $z_2$ being complex conjugates. For multidimensional systems as treated in \cite{LNPRLE} this complex coordinates is more convenient.} \cite{LNPRLE}.


\appendix

\section{Proofs of results in section \ref{sec:mulnoi}}

\subsection{Proof of Lemma \ref{lem:mul:quest:supuntilTisconst}}\label{apsec:prf:lem:mul:quest:supuntilTisconst}
For $\eta \in \C$, by $\eta(\theta)$ we mean $\eta$ evaluated at $\theta \in [-r,0]$. 

Let $\Xeps_{q,t}:=((1-\pi)\hpex{t})(0)$, and $\Xeps_{p,t}:=(\pi\hpex{t})(0).$
And for the unperturbed system \eqref{eq:detDDE}, let $x_{q,t}:=((1-\pi)\pj_tx)(0)$,  and $x_{p,t}:=(\pi\pj_tx)(0)$ with the initial condition $\pj_0x=\hpex{0}=\xi$. Let
$$Z_t:=\int_0^tL_1(\hpex{s}) d{W}_s.$$
 
In \eqref{eq:stabsoldefprob}, $\stabsol$ was defined. Let $$\critsol(t)=(\Th(t)\pih\Ind)(0).$$
Using $\pih \Ind =\Phi \Psiz$ from \eqref{eq:piIdeqPsiz} and $\Th(t)\Phi=\Phi e^{tB}$, we get $\critsol(t)=\Phi(0)e^{tB}\Psiz$.

Using the variation-of-constants formula \eqref{eq:vocformstateW_full}-\eqref{eq:vocformstateW} we have for $t\geq 0$ 
\begin{align*}
\Xeps_t=\Xeps_{p,t}+\Xeps_{q,t}=x_{p,t/\eps^2}+x_{q,t/\eps^2}+D_{p,t}+D_{q,t}+A_{p,t}+A_{q,t}
\end{align*}
with
\begin{align*}
D_{q,t}&:=\int_0^t\stabsol\left(\frac{t-s}{\eps^2}\right)G(\hpex{s}) ds, \qquad A_{q,t}:=\int_0^t\stabsol\left(\frac{t-s}{\eps^2}\right)dZ_s, \\
D_{p,t}&:=\int_0^t\critsol\left(\frac{t-s}{\eps^2}\right)G(\hpex{s}) ds, \qquad A_{p,t}:=\int_0^t\critsol\left(\frac{t-s}{\eps^2}\right)dZ_s.
\end{align*}

For any process $M$, we define $M^*_t:=\sup_{0\leq s\leq t}|M_s|$. Now, what we mean by $D^*_{q,t}$, $A^*_{q,t}$ and $x^*_{q,t}$ etc is clear. Also define
$$\mfXeps_{t}:=\sup_{s\in[0,t]}|\Xeps_{s}|^n.$$
We then have,
\begin{align}\label{eq:fischNapLemm_to_collect}
2^{-5(n-1)}\expt\,\mfXeps_{t}\,\,\quad \leq &\,\,\quad \expt\,|x^*_{p,t/\eps^2}|^n\,\,+\,\,\expt\,|x^*_{q,t/\eps^2}|^n\,\,
\\ &\qquad +\,\,\expt\,|D^*_{p,t}|^n\,\,+\,\,\expt\,|D^*_{q,t}|^n+\,\,\expt\,|A^*_{p,t}|^n+\,\,\expt\,|A^*_{q,t}|^n. \notag
\end{align}

First we focus on the terms involving the process $A$. Using integration by parts we have
\begin{align*}
A_{q,s}=\stabsol(0)Z_s+\int_0^s\eps^{-2}\stabsol'\left(\frac{s-u}{\eps^2}\right)Z_u\,du.
\end{align*}
Using Minkowski inequality,
\begin{align}\label{FmultGcubReq_eq:Dtuseful}
\expt\,|A^*_{q,t}|^n\,\,&\leq\,\,2^{n-1}|\stabsol(0)|^n\expt \sup_{s\in [0,t]}|Z_s|^n\,\,+\,\,2^{n-1}\expt\sup_{s\in [0,t]}\left|\int_0^s\eps^{-2}\stabsol'\left(\frac{s-u}{\eps^2}\right)Z_u\,du\right|^n.
\end{align}
The second term on the RHS of \eqref{FmultGcubReq_eq:Dtuseful} is bounded above (using the exponential decay \eqref{eq:expdecayforhprime}) by
\begin{align*}
2^{n-1}\expt\sup_{s\in [0,t]}\left|\int_0^s\eps^{-2}\left|\stabsol'\left(\frac{s-u}{\eps^2}\right)\right|\,\left|Z_u\right|\,du\right|^n\,\,&\leq\,\,2^{n-1}\expt\sup_{s\in [0,t]}\left|\int_0^s\eps^{-2}\widetilde{K}||L_0||e^{-\kappa(s-u)/\eps^2}\,\left|Z_u\right|\,du\right|^n\\
&\leq\,\,2^{n-1}(\widetilde{K}||L_0||/\kappa)^n\expt \sup_{s\in [0,t]}|Z_s|^n,
\end{align*}
where $\widetilde{K}=K||(1-\pih)\Ind||$.
Hence, {using Burkholder-Davis-Gundy inequality} and Holder inequality
\begin{align*}
\expt\,|A^*_{q,t}|^n\,\,&\leq\,\,2^{n-1}\left(|\stabsol(0)|^n+(\widetilde{K}||L_0||/\kappa)^n\right)\expt \sup_{s\in [0,t]}|Z_s|^n \\
&\leq\,\,2^{n-1}\left(|\stabsol(0)|^n+(\widetilde{K}||L_0||/\kappa)^n\right)C_m\expt \left(\int_0^tL_1^2(\hpex{u})du\right)^{n/2} \\
&\leq\,\,2^{n-1}\left(|\stabsol(0)|^n+(\widetilde{K}||L_0||/\kappa)^n\right)C_{m}t^{\frac{n-2}{2}}||L_1||^n\left(\int_0^t\expt\mfXeps_u du\,+\expt||\xi||^n(t\wedge \eps^2r)\right)\\
&=\,\,C_{m,L}t^{\frac{n-2}{2}}\left(\int_0^t\expt\mfXeps_u du\,+\expt||\xi||^n(t\wedge \eps^2r)\right)
\end{align*}
where $C_{m,L}=2^{n-1}\left(|\stabsol(0)|^n+(\widetilde{K}||L_0||/\kappa)^n\right)C_{m}||L_1||^n$ and $t\wedge \eps^2r$ means $\min\{t,\eps^2r\}$. 

Now we focus on $A_{p,t}$.
\begin{align*}
A_{p,t}\,\,&=\,\, \int_0^t\critsol\left(\frac{t-s}{\eps^2}\right)\,dZ_s \,\,=\,\,\int_0^t\left(\Phi(0)e^{(t-s)B/\eps^2}\Psiz\right)\,dZ_s\,\,=\,\,\Phi(0)e^{tB/\eps^2}\left(\int_0^te^{-sB/\eps^2}dZ_s\right)\Psiz.
\end{align*}
Let $M^c_t=\int_0^t\cos(\om_cs/\eps^2)dZ_s$ and $M^s_t=\int_0^t\sin(\om_cs/\eps^2)dZ_s$. Then
\begin{align*}
\expt |A^*_{p,t}|^n\,\,&\leq\,\, (||\Phi(0)||_2||\Psiz||_2)^n\expt(M^{c,*}_t+M^{s,*}_t)^n.
\end{align*} 
Using BDG and Holder inequalities,
\begin{align*}
\expt |A^*_{p,t}|^n\,\,&\leq\,\, 2^{n-1}(||\Phi(0)||_2||\Psiz||_2)^nCt^{\frac{n-2}{2}}||L_1||^n\left(\int_0^t\expt\mfXeps_u du\,+\expt||\xi||^n(t\wedge \eps^2r)\right).
\end{align*} 

Now we focus on the process $D$.  Using exponential decay of $\stabsol$ we have 
\begin{align*}
|D_{q,t}|\,\,&\leq\,\, \int_0^t\left|\stabsol\left(\frac{t-s}{\eps^2}\right)\right|\,|G(\hpex{s})| ds \\
&\leq\,\,\sup_{s\in[0,t]}|G(\hpex{s})|\int_0^t\widetilde{K}e^{-\kappa(t-s)/\eps^2}ds\,\,\leq\,\,\eps^2(\widetilde{K}/\kappa)\sup_{s\in[0,t]}|G(\hpex{s})|.
\end{align*}
Hence, using the Lipschitz condition $|G(\eta)|\leq K_G||\eta||$ we have,
\begin{align*}
\expt\,|D^*_{q,t}|^n\,\,\leq\,\,\eps^{2n}(\widetilde{K}K_G/\kappa)^n(\expt\mfXeps_{t}+\expt||\xi||^n).
\end{align*}
Now,
\begin{align*}
|D_{p,t}|\,\,&\leq\,\, \int_0^t\left|\critsol\left(\frac{t-s}{\eps^2}\right)\right|\,|G(\hpex{s})| ds \,\,=\,\,\int_0^t\left|\Phi(0)e^{(t-s)B/\eps^2}\Psiz\right|\,|G(\hpex{s})| ds\\
&\leq\,\,||\Phi(0)||_2||\Psiz||_2\int_0^t|G(\hpex{s})| ds \,\,\leq\,\,||\Phi(0)||_2||\Psiz||_2K_G\int_0^t||\hpex{s}|| ds.
\end{align*}
Hence, using BDG and Holder inequalities,
\begin{align*}\expt\,|D^*_{p,t}|^n\,\,\leq\,\,(||\Phi(0)||_2||\Psiz||_2K_G)^nt^{n-1}\left(\int_0^t\expt\mfXeps_{s} ds+\expt||\xi||^n(t\wedge\eps^2r)\right).\end{align*}

Now, we focus on the deterministic terms. Because of our assumption on $L_0$, there exists $C_{L_0}>0$ such that $x^*_{q,t/\eps^2} \leq \sqrt{C_{L_0}}||(1-\pi)\xi||e^{-\kappa t/\eps^2}$ and $x^*_{p,t/\eps^2}\leq \sqrt{C_{L_0}}||\pi\xi||$.

Collecting all the above results in \eqref{eq:fischNapLemm_to_collect}, we have for $n>2$, 
\begin{align*}
(2^{-5(n-1)}-\eps^{2n}(\widetilde{K}K_G/\kappa)^n)\expt\,\mfXeps_{t}\,\,\leq\,\,C_1+C_2\int_0^t\expt\mfXeps_{s} ds,
\end{align*}
where
\begin{align*}
C_1&=C_{L_0}^{n/2}(\expt||\pi\xi||^n+\expt||(1-\pi)\xi||^n)+\eps^{2n}(\widetilde{K}K_G/\kappa)^n\expt||\xi||^n\\
& \qquad \quad + (||\Phi(0)||_2||\Psiz||_2)^n(T^{n-1}K_G^n+2^{n-1}||L_1||^nCT^{(n-2)/2})\expt||\xi||^n\eps^2r,
\end{align*}
$$C_2=(||\Phi(0)||_2||\Psiz||_2K_G)^nT^{n-1}+C_{m,L}T^{(n-2)/2}+2^{n-1}||\Phi(0)||_2^n||\Psiz||_2^nCT^{\frac{n-2}{2}}||L_1||^n.$$
The initial condition $\xi$ is assumed to be deterministic and hence $C_1$ can be written as $C_1=C_{L_0}^{n/2}(||\pi\xi||^n+||(1-\pi)\xi||^n)+\eps^{2n}(KK_G/\kappa)^n||\xi||^n.$

Applying Gronwall inequality we have $\expt\,\mfXeps_{T}\,\,\leq\,\,\frac{2C_1}{2^{-5(n-1)}}\exp\left(\frac{2C_2}{2^{-5(n-1)}}T\right)$ for small enough $\eps$.

For $0\leq n\leq 2$ we can use $\expt \sup_{t\in [0,T]}|\Xeps_t|^n \,\leq\, 1+ \expt \sup_{t\in [0,T]}|\Xeps_t|^3.$

\subsection{Proof of Proposition \ref{prop:mul:supUpsT}}\label{apsec:prf:prop:mul:supUpsT}
Recall the $\stabsol$ defined in \eqref{eq:stabsoldefprob}.
We have
\begin{align*}
\left(\Th(\frac{s-u}{\eps^2})(1-\pih)\Ind\right)(\theta)=\begin{cases}\stabsol\left(\frac{s+\eps^2\theta -u}{\eps^2}\right), \qquad s+\eps^2\theta-u\geq 0\\
\pih \Ind\left(\frac{s+\eps^2\theta -u}{\eps^2}\right), \qquad s+\eps^2\theta-u< 0.
\end{cases}
\end{align*}

Note that
\begin{align}\label{add:newlem:supbound:eq:2_split1}
\sup_{s\in [0,T]}\Upsilon^{\eps}_s \,\,&\leq\,\,\sup_{s\in [0,T]}\sup_{\theta\in [-r,0]} \left|\int_0^{(s+\eps^2\theta)\vee 0}\stabsol\left(\frac{s+\eps^2\theta -u}{\eps^2}\right) dZ_u\right| \\ \notag
& \qquad + \sup_{s\in [0,T]}\sup_{\theta\in [-r,0]} \left|\int_{(s+\eps^2\theta)\vee 0}^s\pih \Ind\left(\frac{s+\eps^2\theta -u}{\eps^2}\right)dZ_u\right| \qquad =: \mathscr{J}_1 + \mathscr{J}_2.
\end{align}
In the above $t\vee s$ means $\max \{t,\,s\}$.
 
For $\mathscr{J}_1$ we have (with $\delta \in (0,2)$)
\begin{align*}
\mathscr{J}_1 \,\,&=\,\,\sup_{t\in [0,T]} \left|\int_0^{t}\stabsol\left(\frac{t -u}{\eps^2}\right)dZ_u\right| \\
&\leq  \,\,\sup_{t\in [r \eps^\delta,T]} \left|\int_0^{t-r \eps^\delta}\stabsol\left(\frac{t -u}{\eps^2}\right)dZ_u\right|   + \sup_{t\in [0,T]} \left|\int_{(t-r \eps^\delta)\vee 0}^t\stabsol\left(\frac{t -u}{\eps^2}\right)dZ_u\right| \qquad =: \mathscr{J}_{1a} + \mathscr{J}_{1b}.
\end{align*}
Using integration by parts and exponential decay of $\stabsol$ and $\stabsol'$ (see \eqref{eq:stabsolexpdecesti}--\eqref{eq:expdecayforhprime}) in $\mathscr{J}_{1a}$ we have
\begin{align*}
\mathscr{J}_{1a} \,\,\,\leq &\,\,\,\sup_{t\in [r \eps^\delta,T]}  |\stabsol(r \eps^{\delta-2}) Z_{t-r \eps^{\delta}}|  + \sup_{t\in [r \eps^\delta,T]} \frac{1}{\eps^2} \int_0^{t-r\eps^\delta}\left|\stabsol'\left(\frac{t -u}{\eps^2}\right)\right|\,|Z_u|\,du \\
\leq & \,\,\, \widetilde{K} e^{-\kappa r \eps^{\delta-2}}\sup_{t\in [0,T]}|Z_t|\,\,+\,\, \sup_{t\in [r \eps^\delta,T]} \frac{1}{\eps^2}||L_0||\widetilde{K}\int_0^{t-r\eps^\delta}e^{-\kappa (t-u)/\eps^2}\,|Z_u|\,du \\
\leq & \,\,\, \widetilde{K} \left(1+\frac{||L_0||}{\kappa}\right)e^{-\kappa r \eps^{\delta-2}}\sup_{t\in [0,T]}|Z_t|,
\end{align*}
where $\widetilde{K}=K||(1-\pi)\Ind||$.
For $\mathscr{J}_{1b}$ we use
\begin{align*}
\stabsol\left(\frac{t -u}{\eps^2}\right)=\stabsol\left(\frac{t -((t-r\eps^{\delta})\vee 0)}{\eps^2}\right)-\frac{1}{\eps^2}\int_{(t-r\eps^{\delta})\vee 0}^u \stabsol'\left(\frac{t -\tau}{\eps^2}\right)d\tau, \qquad \text{ for } u\in [(t-r\eps^{\delta})\vee 0,t].
\end{align*}
Now, using the definition \ref{def:modcont} of modulus of continuity, we have
\begin{align*}
\mathscr{J}_{1b}\,\,&\leq\,\,
\sup_{t\in[0,r\eps^\delta]}|\stabsol(t \eps^{-2})|\,|Z_t| \,+\, \sup_{t\in[r\eps^\delta,T]}|\stabsol(r\eps^{\delta-2})|\,|Z_t-Z_{t-r\eps^\delta}| \\
& \qquad \qquad + \sup_{t\in [0,T]}\frac{1}{\eps^2}\left|\int_{(t-r \eps^\delta)\vee 0}^t\left(\int_{(t-r \eps^\delta)\vee 0}^u\stabsol'\left(\frac{t -\tau}{\eps^2}\right)d\tau\right) dZ_u\right| \\
&\leq 2\widetilde{K} \modcont(r\eps^{\delta},T;Z) + \sup_{t\in [0,T]}\frac{1}{\eps^2}\left|\int_{(t-r \eps^\delta)\vee 0}^t\left(\int_{\tau}^t dZ_u\right) \stabsol'\left(\frac{t -\tau}{\eps^2}\right)d\tau\right| \\
& \leq 2\widetilde{K} \modcont(r\eps^{\delta},T;Z) + \sup_{t\in [0,T]}\frac{1}{\eps^2}\int_{(t-r \eps^\delta)\vee 0}^t |Z_t-Z_{\tau}| \left|\stabsol'\left(\frac{t -\tau}{\eps^2}\right)\right|d\tau \\
& \leq 2\widetilde{K} \modcont(r\eps^{\delta},T;Z) +  \modcont(r\eps^{\delta},T;Z) \sup_{t\in [0,T]}\frac{1}{\eps^2}\int_{(t-r \eps^\delta)\vee 0}^t K||L_0||e^{-\kappa(t-\tau)/\eps^2}d\tau \\
& \leq 2\widetilde{K} \left(1+\frac{||L_0||}{2\kappa}\right)\modcont(r\eps^\delta,T;Z).
\end{align*}

For $\mathscr{J}_2$ we make use of the following facts: 
\begin{align*}
\pih\Ind(v-u)&=\Psiz_1 \cos(\om_c(v-u))+\Psiz_2\sin(\om_c(v-u)) \\
&=(\Psiz_1\cos \om_cv+\Psiz_2\sin \om_cv)\cos \om_cu + (\Psiz_1\sin \om_cv-\Psiz_2\cos \om_cv)\sin \om_cu, 
\end{align*}
and $|\Psiz_1\cos \om_cv+\Psiz_2\sin \om_cv|\leq \sqrt{\Psiz_1^2+\Psiz_2^2}=||\Psiz||_2$. Using these it is easy to see that the  
\begin{align}\label{add:newlem:supbound:eq:2_split2_bound}
\mathscr{J}_2 \,\,\leq \,\, ||\Psiz||_2\,\sup_{s\in [0,T]}\sup_{\theta\in [-r,0]} \left(\left|M^{c,\eps}_{s}-M^{c,\eps}_{(s+\eps^2\theta)\vee 0}\right| + \left|M^{s,\eps}_{s}-M^{s,\eps}_{(s+\eps^2\theta)\vee 0}\right|\right)
\end{align}
where 
\begin{align*}
M^{c,\eps}_t=\int_0^t \cos(\om_cu/\eps^2)dZ_u, \qquad M^{s,\eps}_t=\int_0^t \sin(\om_cu/\eps^2)dZ_u.
\end{align*}
Using the definition \ref{def:modcont} of modulus of continuity, we have
\begin{align*}
\mathscr{J}_2 \,\,\leq \,\,||\Psiz||_2 \,\left( \,\modcont(\eps^2r,T;M^{c,\eps}) \,+\, \modcont(\eps^2r,T;M^{s,\eps}) \,\right).
\end{align*}

Collecting all the above estimates in \eqref{add:newlem:supbound:eq:2_split1} we have
\begin{align*}
\sup_{s\in [0,T]}\Upsilon^{\eps}_s \,\,&\leq\,\,\widetilde{K} \left(1+\frac{||L_0||}{\kappa}\right)e^{-\kappa r \eps^{\delta-2}}\sup_{t\in [0,T]}|Z_t|\\
& \qquad +\,\,2\widetilde{K} \left(1+\frac{||L_0||}{2\kappa}\right)\modcont(r\eps^\delta,T;Z)\\
& \qquad +\,\,||\Psiz||_2 \,\left( \,\modcont(\eps^2r,T;M^{c,\eps}) \,+\, \modcont(\eps^2r,T;M^{s,\eps}) \,\right).
\end{align*}
Now we take expectations.
Using Burkholder-Davis-Gundy inequality and lemma \ref{lem:mul:quest:supuntilTisconst}, we have for $n \geq 1$, $$\expt\sup_{t\in [0,T]}|Z_t|^n\,\leq \,C\expt\la Z\ra_{T}^{n/2}\,\leq\,C\expt\left(T \sup_{t\in[0,T]}|L_1(\pjeps_t\Xeps)|^2\right)^{n/2}\leq CT^{n/2}\mathfrak{C}.$$
Using the  Theorem 1 in section 3 of \cite{FischerNappo} and lemma \ref{lem:mul:quest:supuntilTisconst}, we get that there exists constants $C_{\modcont}$, $C_{\modcont}^c$ and $C_{\modcont}^s$ such that, $\expt \,\modcont^n(r \eps^\delta,T;Z) \leq C_{\modcont}\left({r\eps^\delta}\ln\left(\frac{2T}{r\eps^\delta}\right)\right)^{n/2}$, $\expt \,\modcont^n( \eps^2r,T;M^{c,\eps}) \leq C_{\modcont}^c\left({\eps^2r}\ln\left(\frac{2T}{\eps^2r}\right)\right)^{n/2}$ and $\expt \,\modcont^n( \eps^2r,T;M^{s,\eps}) \leq C_{\modcont}^s\left({\eps^2r}\ln\left(\frac{2T}{\eps^2r}\right)\right)^{n/2}$.
Collecting all, we have
\begin{align*}
2^{-5(n-1)}\expt\sup_{s\in [0,T]}(\Upsilon^{\eps}_s)^n \,\,&\leq\,\,\widetilde{K}^n\left(1+\frac{||L_0||}{\kappa}\right)^nCT^{n/2}\mathfrak{C}\,e^{-n\kappa r \eps^{\delta-2}}\\
& \qquad +\left(2\widetilde{K} \left(1+\frac{||L_0||}{2\kappa}\right)\right)^n C_{\modcont}\left({\eps^\delta r}\ln\left(\frac{2T}{\eps^\delta r}\right)\right)^{n/2} \\
& \qquad +||\Psiz||_2^{n}(C_{\modcont}^c+C_{\modcont}^s)\left({\eps^2r}\ln\left(\frac{2T}{\eps^2r}\right)\right)^{n/2}.
\end{align*}
As $\eps\to 0$, the 2nd term on the RHS dominates and hence we have \eqref{prop:mul:supUpsT_statement_mully}.

\subsection{Proof of Proposition \ref{prop:mul:stabnorm}}\label{apsec:prf:prop:mul:stabnorm}
Using the variation of constants formula \eqref{eq:vocformstateW}, we have
\begin{align}\label{locap03_eq:vocnonlinstable}
||\stbp^{\eps}_s||\,\,\,\leq &\,\,||\int_0^s\Th(\frac{s-u}{\eps^2})(1-\pih)\Ind G(\hpex{u})du || \,\,+\,\,\Upsilon^{\eps}_s.
\end{align}
Using the exponential decay \eqref{eq:expdecesti} we have
\begin{align*}
||\int_0^s & \Th(\frac{s-u}{\eps^2})(1-\pih)\Ind  G(\hpex{u})du || \,\, \leq \,\,\,\int_0^s||\Th(\frac{s-u}{\eps^2})(1-\pih)\Ind  ||\,|G(\hpex{u})|\,du \\ \notag
& \leq \,\,\, K_G\sup_{u\in [0,s]}||\hpex{u}||
\int_0^s \widetilde{K}e^{-\kappa(s-u)/\eps^2}\,du\,\,\,\leq \,\,\,(\eps^2K_G\widetilde{K}/\kappa)\sup_{u\in [0,s]}||\hpex{u}||,
\end{align*}
where $\widetilde{K}=K||(1-\pih)\Ind||$.
Hence for $s\in [0,T]$
\begin{align*}
||\stbp^\eps_s||\,\,&\leq\,\,(\eps^2K_G\widetilde{K}/\kappa)\sup_{u\in [0,s]}||\hpex{u}||+\,\,\Upsilon^{\eps}_s. 
\end{align*}
Hence
\begin{align*}
\sup_{s\in [0,t]}||\stbp^\eps_s||\,\,&\leq\,\,(\eps^2K_G\widetilde{K}/\kappa)\sup_{s\in [0,t]}||\hpex{s}||+\,\,\sup_{s\in [0,t]}\Upsilon^{\eps}_s.
\end{align*}
Raise to power $n$, take expectation and apply lemma \ref{lem:mul:quest:supuntilTisconst} for the first term on the RHS and proposition \ref{prop:mul:supUpsT} for the second term to get \eqref{prop:mul:stabnorm_statement_b}.

\subsection{Proof of lemma \ref{lem:mul:aux4comparison}}\label{apsec:prf:lem:mul:aux4comparison}
For any $\R^2$-vector $v$, and $\theta\in [-r,0]$, we have $\Phi(\theta) e^{tB/\eps^2}v=v_1\cos((\om_ct/\eps^2)+\theta)+v_2\sin((\om_ct/\eps^2)+\theta)$. Hence 
\begin{align}\label{eq:normPhietBepsqv}
||\Phi e^{tB/\eps^2}v||=\sup_{\theta\in [-r,0]}|\Phi(\theta) e^{tB/\eps^2}v| \,\leq\,\sqrt{v_1^2+v_2^2}.
\end{align}
Using Lipshitz condition on $G$, and then using $y^\eps_t-\ydec^\eps_t=\stbp^\eps_t$ and \eqref{eq:normPhietBepsqv}, we get
\begin{align}\label{eq:auxHelpGDiff_boundDrifCoef}
\left|G(\Phi e^{tB/\eps^2}\zed_t+y^\eps_t)-G(\Phi e^{tB/\eps^2}\zedh_t+\ydec^\eps_t)\right|\,\,\leq\,\,K_G(||\stbp^\eps_t||+\sqrt{2\alpha^\eps_t}).
\end{align}
Using the definition \eqref{eq:BigGammaDriftCoefDef} of $\Gamma_t$ we have
\begin{align*}
|\Gamma_t|\,\,&=\,\,\big|\big(\Psiz_1(\zed_t-\zedh_t)_1+\Psiz_2(\zed_t-\zedh_t)_2\big)\cos(\om_ct/\eps^2)+\big(\Psiz_1(\zed_t-\zedh_t)_2-\Psiz_2(\zed_t-\zedh_t)_1\big)\sin(\om_ct/\eps^2)\big| \\
&\leq \,\,\sqrt{(\Psiz_1(\zed_t-\zedh_t)_1+\Psiz_2(\zed_t-\zedh_t)_2)^2+(\Psiz_1(\zed_t-\zedh_t)_2-\Psiz_2(\zed_t-\zedh_t)_1)^2} \\
&= \,\,||\Psiz||_2\,\sqrt{2\alpha^\eps_t}.
\end{align*}
Using the above inequality and \eqref{eq:auxHelpGDiff_boundDrifCoef} in the definition of $\mathscr{B}_t$ we get
\begin{align*}
|\mathscr{B}_t|\,\,&\leq\,\,||\Psiz||_2\sqrt{2\alpha^\eps_t}\,K_G(||\stbp^\eps_t||+\sqrt{2\alpha^\eps_t})\,\,+\,\,\frac12||\Psiz||_2^2||L_1||^2(\sqrt{2\alpha^\eps_t}+||\stbp^\eps_t||)^2 \\
&\leq\,\,||\Psiz||_2K_G(\frac{||\stbp^\eps_t||^2+2\alpha^\eps_t}{2}+2\alpha^\eps_t)\,\,+\,\,||\Psiz||_2^2||L_1||^2(2\alpha^\eps_t+||\stbp^\eps_t||^2) \\
&\leq\,\,C_{\mathscr{B}}(\alpha^\eps_t+||\stbp^\eps_t||^2).
\end{align*}
Using $|\Gamma_t| \leq ||\Psiz||_2\,\sqrt{2\alpha^\eps_t}$ in the definition of $\Sigma_t$ we get
\begin{align*}
\Sigma_t^2\,\,&\leq\,\,\Gamma_t^2||L_1||^2(\sqrt{2\alpha^\eps_t}+||\stbp^\eps_t||)^2\\
&\leq\,\,||\Psiz||_2^2||L_1||^22\alpha^\eps_t(\sqrt{2\alpha^\eps_t}+||\stbp^\eps_t||)^2\,\,\leq\,\,16||\Psiz||_2^2||L_1||^2((\alpha^\eps_t)^2+||\stbp^\eps_t||^4).
\end{align*}


\subsection{Proof of Proposition \ref{prop:mul:alphaissmall}}\label{apsec:prf:prop:mul:alphaissmall}
Using lemma \ref{lem:mul:aux4comparison} we have that $$d\alpha^\eps_t\,\leq\,C_{\mathscr{B}}(\alpha^\eps_t+||\stbp^\eps_t||^2)dt+\Sigma_tdW_t.$$ Let 
\begin{align*}
H_t:=C_{\mathscr{B}}\int_0^t||\stbp^\eps_s||^2ds, \qquad M_t:=\int_0^t\Sigma_sdW_s, \qquad L_t:=\int_0^te^{-C_{\mathscr{B}}s}dM_s.
\end{align*}
Then,
\begin{align}\label{eq:alphaepsIneq}
\alpha^\eps_t\,\leq\,\int_0^tC_{\mathscr{B}}\alpha^\eps_sds\,+\,H_t\,+\,M_t.
\end{align}
Applying Gronwall inequality pathwise, we get,
\begin{align}\label{eq:Gronpathw}
\alpha^\eps_te^{-C_{\mathscr{B}}t}\,\leq\,\,(H_t\,+\,M_t)e^{-C_{\mathscr{B}}t}\,+\,\int_0^t(H_s+M_s)C_{\mathscr{B}}e^{-C_{\mathscr{B}}s}ds.
\end{align}
Using integration by parts we get
\begin{align*}
\int_0^tH_sC_{\mathscr{B}}e^{-C_{\mathscr{B}}s}ds\,\,&=\,\,-H_te^{-C_{\mathscr{B}}t}+\int_0^te^{-C_{\mathscr{B}}s}dH_s \\
&\leq\,\,-H_te^{-C_{\mathscr{B}}t}+\int_0^tdH_s\quad=\quad-H_te^{-C_{\mathscr{B}}t}+H_t.
\end{align*}
Using integration by parts we get $\int_0^tM_sC_{\mathscr{B}}e^{-C_{\mathscr{B}}s}ds\,=\,-M_te^{-C_{\mathscr{B}}t}+L_t.$ Using these results in \eqref{eq:Gronpathw} we get
\begin{align*}
0\,\,\leq\,\,\alpha^\eps_te^{-C_{\mathscr{B}}t}\,\,\leq\,\,L_t+H_t.
\end{align*}
Note that $L$ is a martingale. We have
\begin{align*}
\expt \sup_{s\in [0,t]}\left(\alpha^\eps_se^{-C_{\mathscr{B}}s}\right)^2\,\,\leq\,\, \expt \sup_{s\in [0,t]} (L_s+H_s)^2 \,\,&\leq\,\, 2\expt \sup_{s\in [0,t]} L_s^2 + 2\expt \sup_{s\in [0,t]} H_s^2 \\
&\leq\,\, 8\expt  L_t^2 + 2\expt H_t^2
\end{align*}
where in the last step we have used Doob's $L^p$ inequality (Theorem 2.1.7 in \cite{RevuzYor}) and the fact that $H$ is non-decreasing.
Now, using BDG inequality
\begin{align*}
\expt L_{t}^2\,\,&=\,\,\expt\int_0^{t}e^{-2C_{\mathscr{B}}s}\Sigma_s^2ds\,\,\leq\,\,C_{\Sigma}\expt\int_0^{t}e^{-2C_{\mathscr{B}}s}((\alpha^\eps_s)^2+||\stbp^\eps_s||^4)ds,\\
&\leq\,\,C_{\Sigma}\int_0^{t}\expt \sup_{u\in [0,s]}\left(\alpha^\eps_ue^{-C_{\mathscr{B}}u}\right)^2ds\,\,+\,\,C_{\Sigma}\int_0^{t}\expt||\stbp^\eps_s||^4ds.
\end{align*}
Using Holder inequality we have
$$2\expt H_t^2\,\,=\,\,2\expt\left(C_{\mathscr{B}}\int_0^t||\stbp^\eps_s||^2ds\right)^2\,\,\leq\,\,2C_{\mathscr{B}}^2t\int_0^t\expt||\stbp^\eps_s||^4ds.$$
Hence,
\begin{align*}
\expt \sup_{s\in [0,t]}\left(\alpha^\eps_se^{-C_{\mathscr{B}}s}\right)^2\,\,\leq\,\,8C_{\Sigma}\int_0^{t}\expt \sup_{u\in [0,s]}\left(\alpha^\eps_ue^{-C_{\mathscr{B}}u}\right)^2ds\,\,+\,\,(8C_{\Sigma}+2C_{\mathscr{B}}^2t)\int_0^{t}\expt||\stbp^\eps_s||^4ds.
\end{align*}
Using Gronwall and then \eqref{prop:mul:stabnorm_statement_b} we have
\begin{align*}
\expt \sup_{s\in [0,T]}\left(\alpha^\eps_se^{-C_{\mathscr{B}}s}\right)^2\,\,\leq&\,\,(8C_{\Sigma}+2C_{\mathscr{B}}^2T)T\,2^6\left(\eps^{8}2^{3}\left(\frac{K_GK}{\kappa}\right)^4\mathfrak{C}^4\,+\,2^{3}\expt\sup_{s\in [0,T]}(\Upsilon^{\eps}_s)^4\right)e^{8C_{\Sigma}\,T}\\
\leq&\,\,C\left(r\eps^\delta\ln(\frac{2T}{r\eps^\delta})\right)^2, \qquad \text{for small enough } \eps. 
\end{align*}
Hence
\begin{align*}
\expt \sup_{s\in [0,T]}\left(\alpha^\eps_s\right)^2\,\,
\leq\,\,Ce^{2C_{\mathscr{B}}T}\left(r\eps^\delta\ln(\frac{2T}{r\eps^\delta})\right)^2, \qquad \text{for small enough } \eps. 
\end{align*}


\subsection{Proof of Proposition \ref{prop:mul:betaissmall}}\label{apsec:prf:prop:mul:betaissmall}
Following exactly the same technique as for $\alpha^\eps$, we arrive at
\begin{align*}
\expt \sup_{s\in [0,t]}\left(\beta_se^{-C_{\mathscr{B}}s}\right)^2\,\,\leq\,\,8C_{\Sigma}\int_0^{t}\expt \sup_{u\in [0,s]}\left(\beta_ue^{-C_{\mathscr{B}}u}\right)^2ds\,\,+\,\,(8C_{\Sigma}+2C_{\mathscr{B}}^2t)\int_0^{t}\expt||\ydec^\eps_s||^4ds.
\end{align*}
Using the exponential decay \eqref{eq:stabsolexpdecesti} we have that $\int_0^{t}\expt||\ydec^\eps_s||^4ds  \,\leq\,\widetilde{K}^4\int_0^te^{-4\kappa s/\eps^2}ds \,\leq\,\eps^2(\widetilde{K}^4/4\kappa)$ where $\widetilde{K}=K||(I-\pih)\Ind||$. Using Gronwall inequality we have
\begin{align*}
\expt \sup_{s\in [0,T]}\left(\beta_se^{-C_{\mathscr{B}}s}\right)^2\,\,&\leq\,\,\eps^2(8C_{\Sigma}+2C_{\mathscr{B}}^2T)T (\widetilde{K}^4/4\kappa) e^{8C_{\Sigma}\,T}.
\end{align*}

\subsection{Proof of theorem \ref{prop:mul:finalthem}}\label{apsec:prf:prop:mul:mainTh}
Using $X^\eps(t)=\Phi(0)e^{tB/\eps^2}\zed_t+y^\eps_t(0)$ and Minkowski inequality in \eqref{eq:mul:finthem_zedtapprx} and then using \eqref{eq:normPhietBepsqv} we have
\begin{align*}
\big|X^\eps(t)-\big(\Phi(0)e^{tB/\eps^2}\zedt_t+\ydec^\eps_t(0)\big)\big|^4 \,&\leq\, 8\left(||\Phi e^{tB/\eps^2}(\zed_t-\zedh_t)||^4+||\Phi e^{tB/\eps^2}(\zedh_t-\zedt_t)||^4+||y^\eps_t-\ydec^\eps_t||^4\right) \\
&\leq\, 8\left(||\zed_t-\zedh_t||_2^4+||\zedh_t-\zedt_t||_2^4+||\stbp^\eps_t||^4\right) \\
&\leq\, 8\left(4(\alpha^\eps_t)^2+4(\beta^\eps_t)^2+||\stbp^\eps_t||^4\right).
\end{align*}
Combining propositions \ref{prop:mul:stabnorm}, \ref{prop:mul:alphaissmall} and \ref{prop:mul:betaissmall} and realizing that  $\left(r\eps^\delta\ln(\frac{2T}{r\eps^\delta})\right)^2 \,\ll\,\eps^2$ for small enough $\eps$ when $\delta \in (1,2)$, we get \eqref{eq:mul:finthem_zedtapprx}. Similar is the proof for \eqref{eq:mul:finthem_zedhapprx}.

\subsection{Proof of lemma \ref{lem:mul:auxzbound}}\label{apsec:prf:lem:mul:auxzbound}
Define $\zeta^\eps_t=\frac12||\zedt_t||_2^2$. Using Ito formula we have $d\zeta^\eps_t=\widetilde{\mathscr{B}}_tdt+\widetilde{\Sigma}_tdW_t$ where
\begin{align*}
\widetilde{\mathscr{B}}_t\,\,=\,\,\widetilde{\Gamma}_tG(\Phi e^{tB/\eps^2}\zedt_t)\,+\, \frac12||e^{-tB/\eps^2}\Psiz||_2^2 \big(L_1(\Phi e^{tB/\eps^2}\zedt_t)\big)^2, \qquad \quad \widetilde{\Sigma}_t=\widetilde{\Gamma}_tL_1(\Phi e^{tB/\eps^2}\zedt_t),
\end{align*}
and $\widetilde{\Gamma}_t=\sum_{i=1}^2(\zedt_t)_i(e^{-tB/\eps^2}\Psiz)_i$. Using similar technique as in proof of lemma \ref{lem:mul:aux4comparison} it can be shown that $|\widetilde{\mathscr{B}}_t| \leq C_{\widetilde{\mathscr{B}}}\zeta^\eps_t$ 
and $\widetilde{\Sigma}_t^2\leq C_{\widetilde{\Sigma}}(\zeta^\eps_t)^2$ where $C_{\widetilde{\mathscr{B}}}=2||\Psiz||_2K_G+||\Psiz||_2^2||L_1||^2$ and $C_{\widetilde{\Sigma}}=4||\Psiz||_2^2||L_1||^2$. Hence we have
\begin{align*}
\zeta^\eps_t \,\,\leq\,\, \int_0^tC_{\widetilde{\mathscr{B}}}\zeta^\eps_sds + \widetilde{H}_t + \widetilde{M}_t, \qquad \widetilde{H}_t:=\zeta^\eps_0, \quad \widetilde{M}_t:=\int_0^t \widetilde{\Sigma}_sdW_s,
\end{align*}
which is analogous to \eqref{eq:alphaepsIneq}. Following the same technique as in section \ref{apsec:prf:prop:mul:alphaissmall} we get
\begin{align*}
\expt \sup_{s\in [0,T]}\left(\zeta^\eps_s\right)^2\,\,
\leq\,\,\expt (\zeta^\eps_0)^2 e^{(2C_{\widetilde{\mathscr{B}}}+8C_{\widetilde{\Sigma}})T}.
\end{align*}
\section{Proofs of results in section \ref{sec:addnoi}}

\subsection{Proof of proposition \ref{prop:add:stabnorm}}\label{apsec:prf:prop:add:stabnorm}
Using the variation of constants formula \eqref{eq:vocformstateW_addn} and definition \eqref{eq:UpsandZdef_add}, we have
\begin{align}\label{locadd:eq:vocnonlinstable}
||\stbp^{\eps}_s||\,\,\,\leq &\,\,\,||\int_0^s\Th(\frac{s-u}{\eps^2})(1-\pih)\Ind G(\pjeps_u \Xeps)du || \,\,+\,\,\Upsilon^{\eps}_s.
\end{align}
For $G$ defined in \eqref{eq:considerMAINadd} we have
\begin{align}\label{eq:boundonGforTstopt}
|G(\eta)| &\leq \int |\pi\eta| |d\nu_1|+\int |(1-\pi)\eta| |d\nu_1|+ \sum_{j=0}^3\binom{3}{j}\int_{-r}^0|\pi\eta|^{3-j}|(1-\pi)\eta|^j|d\nu_3|.
\end{align}
For $s\in [0,T\wedge\stopteps]$ we have that $||\pi\pjeps_s\Xeps||\leq C_{\stopt}$. Using this fact and $||(1-\pi)\pjeps_s \Xeps||\leq ||\stbp^\eps_s||+||\ydec^\eps_s||$ in \eqref{eq:boundonGforTstopt}, and using inequalities $q \leq 1+q^3$, $q^2\leq 1+q^3$ for $q>0$; 
we have for $s\in [0,T\wedge\stopteps]$
$$|G(\pjeps_s\Xeps)|\leq C(1+||\stbp^\eps_s||^3+||\ydec^\eps_s||^3).$$
This $C$ is of the order of $C_{\stopt}^3$ for large $C_{\stopt}$.
Now, using the above inequality and the exponential decays \eqref{eq:expdecesti} and \eqref{eq:ydecexpfastdec} we have
\begin{align*}
||\int_0^s & \Th(\frac{s-u}{\eps^2})(1-\pih)\Ind  G(\hpex{u})du || \,\, \leq \,\,\,\int_0^s||\Th(\frac{s-u}{\eps^2})(1-\pih)\Ind  ||\,|G(\hpex{u})|\,du \\ \notag
& \leq \,\,\, C\int_0^s e^{-\kappa(s-u)/\eps^2}(1+||\stbp^\eps_u||^3+||\ydec^\eps_u||^3)\,du \,\,\,\,\\
&\leq \,\,\,(C\eps^2/\kappa)(1+K^3||(1-\pi)\pjeps_0\Xeps||^3/2)\,\,+\,\,C\int_0^s e^{-\kappa(s-u)/\eps^2}\,||\stbp^\eps_u||^3\,du.
\end{align*} 
Plugging the above inequality in \eqref{locadd:eq:vocnonlinstable} we have for $s\in [0,T\wedge \stopteps]$
\begin{align*}
||\stbp^\eps_s||-\left({C}\eps^2\,\,+\,\,{C}\int_0^s e^{-\kappa(s-u)/\eps^2}\,||\stbp^\eps_u||^3\,du\right)\,\,\leq\,\,\Upsilon^{\eps}_s,
\end{align*}
where $C$ above is of the order of $C_{\stopt}^3$ for large $C_{\stopt}$.
For the  RHS of the above inequality we use Markov inequality, i.e.
\begin{align*}
\mbbP\left[\sup_{s\in [0,T\wedge\stopteps]}\Upsilon^{\eps}_s\,\geq\,\eps^a\right] \,\,\leq\,\,\eps^{-a}\expt\left[\sup_{s\in [0,T]}\Upsilon^{\eps}_s\right]
\end{align*}
and then proposition \ref{prop:add:supUpsT}. Then we have the following statement: 

Fix $a\in [0,1)$. For $\delta \in (2a,2)$, there exists constants $\hat{C}>0$ (independent of $\delta$ and $a$) and $\eps_{\delta}>0$, such that for $\eps < \eps_{\delta}$ 
\begin{align*}
\mbbP\bigg[\forall s\in [0,T\wedge\stopteps],\quad ||\stbp^\eps_s|| \,\,\leq\,\,& C\eps^2\,\,+\,\,C\int_0^s e^{-\kappa(s-u)/\eps^2}\,||\stbp^\eps_u||^3\,du+ 2\eps^a \bigg] \\ 
& \geq\,\,\,1-\hat{C}\eps^{-a}\sqrt{{r\eps^\delta}\ln\left(\frac{T}{r\eps^\delta}\right)}. 
\end{align*}

Using Gronwall kind of inequality (Theorem 2.4.8 in \cite{Bainov}) we have that LHS of above inequality is bounded above by
\begin{align*}
\mbbP\bigg[\forall s\in [0,T\wedge\stopteps],\quad ||\stbp^\eps_s|| \,\,\leq\,\,\frac{C\eps^2+2\eps^a }{\sqrt{1-2\int_0^s\left(C\eps^2+2\eps^a \right)^2Cdu}} \bigg] 
\end{align*}
which is bounded above by (for small enough $\eps$, i.e. $\eps\ll (2/C)^{1/(2-a)}$)
\begin{align*}
\mbbP\bigg[\forall s\in [0,T\wedge\stopteps],\quad ||\stbp^\eps_s|| \,\,\leq\,\,\frac{4\eps^a }{\sqrt{1-2CT(4\eps^a)^2}} \bigg]. 
\end{align*}
which is bounded above by (for small enough $\eps$, i.e. $\eps \ll (1/C)^{1/2a}$)
\begin{align*}
\mbbP\bigg[\forall s\in [0,T\wedge\stopteps],\quad ||\stbp^\eps_s|| \,\,\leq\,\,{8\eps^a } \bigg]. 
\end{align*}
Hence \eqref{add:newprop:supboundGron_statement_Gcub_lin_a_add} follows.

\subsection{Proof of lemma \ref{lem:add:aux4comparison}}\label{apsec:prf:lem:add:aux4comparison}
Recall that $G(\eta)=\int_{-r}^0\eta(\theta)d\nu_1(\theta)+\int_{-r}^0\eta^3(\theta)d\nu_3(\theta)$. For brevity, let $\bfe$ denote $e^{tB/\eps^2}$.  Now,
\begin{align*}
\bigg|\int_{-r}^0&(\Phi \bfe \zed_t+y^\eps_t)^3d\nu_3  - \int_{-r}^0(\Phi \bfe\zedh_t+\ydec^\eps_t)^3d\nu_3 \bigg|\\
&=\,\left|\int_{-r}^0(\Phi \bfe\zed_t+\ydec^\eps_t+\stbp^\eps_t)^3d\nu_3 - \int_{-r}^0(\Phi \bfe \zed_t+\Phi \bfe(\zedh_t-\zed_t)+\ydec^\eps_t)^3d\nu_3 \right| \\
&\leq \sum_{j=1}^3\binom{3}{j}\left|\int_{-r}^0(\Phi \bfe\zed_t)^{3-j}((\ydec^\eps_t+\stbp^\eps_t)^j-(\Phi \bfe(\zedh_t-\zed_t)+\ydec^\eps_t)^j)d\nu_3 \right| \\
& \leq \left(3\sum_{j=0}^2||\Phi \bfe\zed_t||^j\right)\left(3\sum_{j=0}^2||\ydec^\eps_t||^j\right) \left(\int|d\nu_1|+|d\nu_3|\right)\sum_{j=1}^3(||\stbp^\eps_t||^j+||\Phi \bfe(\zedh_t-\zed_t)||^j).
\end{align*}
Note that for $t\in [0,T\wedge\stopteps]$, $\zed$ is bounded. Also, due to the exponential decay \eqref{eq:ydecexpfastdec} we have that $||\ydec^\eps_t|| \,<\,K||(I-\pi)\hpex{0}||$.  Hence we have,
\begin{align*}
\bigg|\int_{-r}^0(\Phi \bfe \zed_t+y^\eps_t)^3d\nu_3  - \int_{-r}^0(\Phi \bfe\zedh_t+\ydec^\eps_t)^3d\nu_3 \bigg| 
\,\, \leq \,\,C\sum_{j=1}^3(||\stbp^\eps_t||^j+(\sqrt{2\alpha^\eps_t})^j),
\end{align*}
where $C$ in the above inequality is of the order of $C_{\stopt}^2$ for large $C_{\stopt}$.

Similarly,
\begin{align*}
\left|\int_{-r}^0(\Phi \bfe \zed_t+y^\eps_t)d\nu_1 - \int_{-r}^0(\Phi \bfe\zedh_t+\ydec^\eps_t)d\nu_1 \right|\,&\leq\,C(||\stbp^\eps_t||+(\sqrt{2\alpha^\eps_t})).
\end{align*}
Combining, we get that for $t\in [0,T\wedge\stopteps]$, $$|\Gamma_t|\left|G(\Phi e^{tB/\eps^2}\zed_t+y^\eps_t)-G(\Phi e^{tB/\eps^2}\zedh_t)\right|\,\,\leq\,\,|\Gamma_t|C\sum_{j=1}^3(||\stbp^\eps_t||^j+(\sqrt{2\alpha^\eps_t})^j).$$
We have shown $|\Gamma_t|\leq ||\Psiz||_2\sqrt{2\alpha^\eps_t}$ in section \ref{apsec:prf:lem:mul:aux4comparison}. Hence, if we define $\mathfrak{B}$ by \eqref{eq:def:mfBadd}, then we have, $|\mathscr{B}_t|\leq \mathfrak{B}(\alpha_t,||\stbp^\eps_t||)$ for  $t\in[0,T\wedge \stopteps]$.


\subsection{Proof of proposition \ref{prop:add:alhaissmall}}\label{apsec:prf:prop:add:alhaissmall}
Using lemma  \ref{lem:add:aux4comparison} and $\sqrt{2\alpha}(\sqrt{2\alpha})^2\leq 4\alpha(1+\alpha)$ we have
$$d\alpha^\eps_t \,\,\leq\,\,(6C\alpha^\eps_t+8C(\alpha^\eps_t)^2)dt + C\sqrt{2\alpha^\eps_t}(\sum_{j=1}^3||\stbp^\eps_t||^j)dt, \qquad t\in [0,T\wedge\stopteps], \qquad \alpha^\eps_0=0,$$
where $C$ is from lemma \ref{lem:add:aux4comparison}. This $C$ is  of the order $\Ord(C_{\stopt}^2)$ for large $C_{\stopt}$.
Let $\eps_{a,C_{\stopt}}$ be as in proposition \ref{prop:add:stabnorm} and define $\eps_{a,C_{\stopt},1}=\min\{1,\eps_{a,C_{\stopt}}\}$. Then we have, for $\eps < \min\{\hat{\eps}_{\delta},\,\eps_{a,C_{\stopt},1}\}$, with probability atleast $p_{\eps}:=1-\hat{C}\eps^{-a}\sqrt{{r\eps^\delta}\ln\left(\frac{T}{r\eps^\delta}\right)},$
$$\sum_{j=1}^3||\stbp^\eps_t||^j \,\,\leq\,\,24\eps^a, \qquad \forall t\in [0,T\wedge\stopteps].$$
(we have used that for $\eps \leq 1$, $\eps^{3a}\leq \eps^a$.)

Let $\mathfrak{s}:=\inf\{t\geq 0\,:\,\alpha^\eps_t\geq 1\}$. 
Using $\sqrt{2\alpha}\leq 2(1+\alpha)$, and $\alpha^2<\alpha$ when $\alpha<1$; we have for $t\in [0,T\wedge\stopteps\wedge \mathfrak{s}]$ 
\begin{align*}
\frac1C d\alpha^\eps_t \,\,\leq\,\,\alpha^\eps_t\left(6+8+48\eps^a\right)dt\,+\,48\eps^adt
\end{align*}
Using Gronwall we get for $t\in [0,T\wedge\stopteps\wedge \mathfrak{s}]$ 
\begin{align*}
\alpha^\eps_t\,\,\leq\,\,\frac{48\eps^a}{14+48\eps^a}(e^{(14+48\eps^a)Ct}-1).
\end{align*}
Define $\eps_{a,C_{\stopt},2}=\min\{48^{-1/a},\,(14e^{-15CT}/48)^{2/a}\}$. Then for $\eps < \min\{\hat{\eps}_{\delta},\,\eps_{a,C_{\stopt},1},\,\eps_{a,C_{\stopt},2}\}$ we have $\alpha^\eps_t \leq \eps^{a/2}$ for $t\in [0,T\wedge\stopteps\wedge \mathfrak{s}]$. But, since $\eps^{a/2}<1$ we have $\mathfrak{s}> T\wedge\stopteps$ and hence $\alpha^\eps_t \leq \eps^{a/2}$ for $t\in [0,T\wedge\stopteps]$.

Note that $C$ is of the order $\Ord(C_{\stopt}^2)$ and hence $\eps_{a,C_{\stopt},2}$ is of the order $\Ord(e^{-30C_{\stopt}^2T/a})$.


\subsection{Proof of theorem \ref{prop:helpinprobconvg}}\label{apsec:prf:prop:add:helpinprobconvg}

Since the $\zedh$ system possesses the property $\mathscr{P}(T,q)$,  $\,\,\exists\, C_{\stopt},\,\eps_*>0$ such that $\forall \eps<\eps_*$, we have $\mbbP[E^\eps]\geq 1-q$ where $E^\eps$ is given in \eqref{assonzedh_setEeps}. 

The stopping time $\stopteps$ was defined at \eqref{eq:stoptimedefy}.

Using 
\begin{align*}
\Om \,\,&=\,\,E^\eps \,\cup\, (\Om \setminus E^\eps ) \\
&=\,\,(E^\eps \cap \{\stopteps \leq T\}) \,\cup\, (E^\eps \cap \{\stopteps > T\}) \,\cup\, (\Om \setminus E^\eps ), 
\end{align*}
we have
\begin{align}\label{eq:auxlevprob2remstopt_add}
H^\eps \,\,&=\,\,(E^\eps \cap \{\stopteps \leq T\}\cap H^\eps) \,\cup\, (E^\eps \cap \{\stopteps > T\}\cap H^\eps) \,\cup\, (H^\eps \cap (\Om \setminus E^\eps) ).
\end{align}

Now we deal with the first term on the RHS of \eqref{eq:auxlevprob2remstopt_add}. Note that
\begin{align}
\sup_{t\in[0,T\wedge\stopteps]}\sqrt{2\alpha^{\eps}_t} \,\,\geq\,\sup_{t\in[0,T\wedge\stopteps]}||\Phi e^{tB/\eps^2}(\zed_t-\zedh_t)||\,\,\geq\,\sup_{t\in[0,T\wedge\stopteps]}||\Phi e^{tB/\eps^2}\zed_t||-\sup_{t\in[0,T\wedge\stopteps]}||\Phi e^{tB/\eps^2}\zedh_t||.
\end{align}
In $E^\eps$ we have $\sup_{t\in[0,T\wedge\stopteps]}||\Phi e^{tB/\eps^2}\zedh_t||<0.99C_{\stopt}$, and in $\{\stopteps \leq T\}$ we have $\sup_{t\in[0,T\wedge\stopteps]}||\Phi e^{tB/\eps^2}\zed_t||\geq C_{\stopt}$. Hence, in $E^\eps \cap \{\stopteps \leq T\}$ we have that $\sup_{t\in[0,T\wedge\stopteps]}\sqrt{2\alpha^{\eps}_t}>0.01C_{\stopt}$. Hence $E^\eps \cap \{\stopteps \leq T\} \,\subset\,J^\eps$ where
$J^\eps:=\left\{\om\,:\,\sup_{t\in[0,T\wedge\stopteps]}\alpha^{\eps}_t \,\geq\,\frac12(0.01C_{\stopt})^2\right\}$.
By proposition \ref{prop:add:alhaissmall}, $\exists \,\eps_1$ such that $\forall \eps <\eps_1$, $\mbbP[J^\eps]<1-p_{\eps}$.

Now we deal with the second term on the RHS of \eqref{eq:auxlevprob2remstopt_add}.
Note that  $\{\stopteps > T\}\cap H^\eps \,\subset\,\widetilde{J}^\eps$ where
$$\widetilde{J}^\eps:=\left\{\om\,:\,\sup_{t\in[0,T\wedge\stopteps]}\alpha^{\eps}_t \,\geq\,\eps^{a/2}\right\}.$$
By proposition \ref{prop:add:alhaissmall}, $\exists \,\eps_2$ such that $\forall \eps <\eps_2$, $\mbbP[\widetilde{J}^\eps]<1-p_{\eps}$.

And $\forall \eps <\eps_*$ $\mbbP[\Om \setminus E^\eps]<q$.

 Combining we have, when $\eps<\min\{\eps_1,\eps_2,\eps_*\}=:\eps_q$, $$\mbbP[H^\eps]<q+2(1-p_{\eps}).$$

Note that \eqref{eq:auxlevprob2remstopt_add} is true with $H^\eps$ replaced by $S^\eps$. Using that $$\{\stopteps > T\}\cap S^\eps \,\subset\,\{\om:\sup_{t\in [0,T\wedge\stopteps]}||\stbp^\eps_t||\geq 8\eps^a\}$$
and also that the probability of the latter set is bounded above by $1-p_{\eps}$ we get the desired result.

\subsection{Proof of proposition \ref{lem:add:propertyPTq}}\label{appsec:lem:add:propertyPTq}
We have
\begin{align}\label{eq:lem:add:propertyPTq:eq1}
\zedh_t=\zedh_0+\int_0^te^{-sB/\eps^2}\Psiz G(\Phi e^{sB/\eps^2}\zedh_s+\ydec^\eps_s)ds+\mfw_t, \qquad \mfw_t:=\int_0^te^{-sB/\eps^2}\Psiz \sigma dW_s.
\end{align}
To keep things simple, we prove assuming $||\ydec^\eps_0||=0$ (which ensures that $||\ydec^\eps_t||=0$ for all $t\geq 0$).  
Using $\int_0^t ||\ydec^\eps_s||^nds \leq \eps^2(K/nk)||\ydec^\eps_0||^n$ (because of exponential decay \eqref{eq:ydecexpfastdec}), it is easy to see that the following ideas work even if we assume that  $||\ydec^\eps_0||\neq 0$ (we assume the initial condition is deterministic).

We will make use of the inequality\footnote{Note that $||\Phi e^{tB/\eps^2}v||\,=\,\sup_{\theta\in [-r,0]}|v_1\cos(\theta+ \om_ct/\eps^2)+v_2\sin(\theta + \om_ct/\eps^2)| \,\leq\,|v_1|+|v_2|.$} that for $\R^2$ vector $v$, 
\begin{align}\label{eq:Phievleqv1}
||\Phi e^{tB/\eps^2}v|| \leq ||v||_1
\end{align}
 where $||\cdot||_1$ indicates the 1-norm. 
Using the structure of $G$ specified at \eqref{eq:considerMAINadd} in \eqref{eq:lem:add:propertyPTq:eq1} we have (with some $K_G>0$)
\begin{align}\label{eq:lem:add:propertyPTq:auxeq1p5}
||\zedh_t||_1\leq||\zedh_0||_1+||\Psiz||_1\int_0^tK_G (||\zedh_s||_1+||\zedh_s||_1^3)ds+||\mfw_t||_1.
\end{align}
Because the initial condition is deterministic, we have a $C_0>0$ such that $||\zedh_0||_1<C_0$. For any $C_a>4C_0$, define $T_{C_a}:=\left(2(1+C_a^2)K_G||\Psiz||_1\right)^{-1}.$

Suppose that $\sup_{t\in [0,T]}||\mfw_t||_1<C_a/4$. If $T\leq T_{C_a}$, as long as $||\zedh_t||_1<C_a$, we have (using \eqref{eq:lem:add:propertyPTq:auxeq1p5}) for $t\in [0,T]$
\begin{align*}
||\zedh_t||_1\,\,&<\,\, C_0+K_G||\Psiz||_1(C_a+C_a^3)T+\frac14C_a \\
&< \,\,\frac14C_a +K_G||\Psiz||_1(C_a+C_a^3)T_{C_a}+\frac14C_a
\,\,=\,\,C_a.
\end{align*}
This means that, if $C_a>4C_0$ and $T\leq T_{C_a}$, then we have $\sup_{t\in[0,T]}||\zedh_t||_1<C_a$ provided $\sup_{t\in [0,T]}||\mfw_t||_1<C_a/4$.

Hence, for $C_a>4C_0$ and $T\leq T_{C_a}$,
\begin{align}\label{eq:lem:add:propertyPTq:auxeq2}
\mbbP\left[\sup_{t\in[0,T]}||\zedh_t||_1 \geq C_a\right]\,\,\leq\,\,\mbbP\left[\sup_{t\in[0,T]}||\mfw_t||_1 \geq C_a/4\right].
\end{align}
Using Markov inequality and Burkholder-Davis-Gundy inequality we have
\begin{align*}
\mbbP[\sup_{t\in[0,T]}||\mfw_t||_1\,\geq\,C_a/4]\,\,&\leq\,\,\frac{\expt\sup_{t\in [0,T]}||\mfw_t||_1}{C_a/4}\\
&\leq\,\, \frac{\sum_{j=1}^2C_{bdg}\expt \sqrt{\int_0^T (e^{-sB/\eps^2}\Psiz \sigma)^2ds}}{C_a/4} \,\,\leq\,\,\frac{8|\sigma|\,||\Psiz||_2\,\sqrt{T}C_{bdg}}{C_a}.
\end{align*}
Using the above inequality in \eqref{eq:lem:add:propertyPTq:auxeq2} we have for $C_a>4C_0$ and  $T \leq T_{C_a}$
\begin{align}\label{eq:lem:add:propertyPTq:auxeq3}
\mbbP\left[\sup_{t\in[0,T]}||\zedh_t||_1 \geq C_a\right]\,\,\leq\,\,\frac{8|\sigma|\,||\Psiz||_2C_{bdg}}{C_{a}\sqrt{2K_G||\Psiz||_1(1+C_a^2)}}\,\,=:\,\,f(C_a).
\end{align}
Given $q>0$, let $C_{a,q}>4C_0$ be such that $f(C_a)<q$, $\forall C_a \geq C_{a,q}$.
Such a $C_{a,q}$ exists because $f$ is monotonically decreasing in $C_a$. Set $T_q=T_{C_{a,q}}$. 
Choose $C_{\stopt}>C_{a,q}/0.99$. Let 
\begin{align}\label{eq:def:ETqwidetilde}
\widetilde{E}^\eps := \left\{\om\,:\,\sup_{t\in [0,T_q]}||\Phi e^{tB/\eps^2}\zedh_t||\,<\,0.99C_{\stopt}\right\}.
\end{align}
Now, using \eqref{eq:Phievleqv1} and \eqref{eq:lem:add:propertyPTq:auxeq3}
\begin{align*}
\mbbP[\Om \setminus \widetilde{E}^\eps]\,\,& =  \mbbP\left[\sup_{t\in [0,T_q]}||\Phi e^{tB/\eps^2}\zedh_t||\,\geq\,0.99C_{\stopt}\right]\\
&\leq  \mbbP\left[\sup_{t\in [0,T_q]}||\zedh_t||_1\,\geq\,0.99C_{\stopt}\right]\,\, \leq\,\,\mbbP\left[\sup_{t\in[0,T_q]}||\zedh_t||_1\geq C_{a,q}\right]\,\,\leq \,\,f(C_{a,q})\,\,<\,\,q.
\end{align*}
Hence $\mbbP[\widetilde{E}^\eps] \geq 1-q$. But, for $T \leq T_q$ the set $E^\eps$ defined in \eqref{assonzedh_setEeps} contains $\widetilde{E}^\eps$ and hence  we have that for $T\in [0,T_q]$, $\mbbP[E^\eps]\,\,\geq\,\,1-q.$ Hence \eqref{eq:considerMAINadd} possesses the property $\mathscr{P}(T,q)$ for $T\in [0,T_q]$.

\subsection{Proof of proposition \ref{lem:add:propertyPTq_stable}}\label{appsec:lem:add:propertyPTq_stableNony}

To keep things simple, we prove assuming $||\ydec^\eps_0||=0$ (which ensures that $||\ydec^\eps_t||=0$ for all $t\geq 0$).  
Using $\int_0^t ||\ydec^\eps_s||^nds \leq \eps^2(K/nk)||\ydec^\eps_0||^n$ (because of exponential decay \eqref{eq:ydecexpfastdec}), it is easy to see that the following ideas work even if we assume that  $||\ydec^\eps_0||\neq 0$.

For simplicity of notation we write $G=G_1+G_3$ where $G_1$ is the linear part and $G_3$ is the cubic part.

We have
\begin{align*}
\zedh_t=\zedh_0+\int_0^te^{-sB/\eps^2}\Psiz G(\Phi e^{sB/\eps^2}\zedh_s)ds+\mfw_t, \qquad \mfw_t:=\int_0^te^{-sB/\eps^2}\Psiz \sigma dW_s.
\end{align*}
Writing $\zpw_t=\zedh_t-\mfw_t$, we have
\begin{align}\label{eq:howzpwevolves}
\dot{\zpw}_t=e^{-tB/\eps^2}\Psiz G(\Phi e^{tB/\eps^2}(\zpw_t+\mfw_t))
\end{align}
from which we can write (using that the transpose of $e^{tB/\eps^2}$ is $e^{-tB/\eps^2}$)
\begin{align*}
\frac12\frac{d}{dt}||\zpw_t||_2^2\,&=\,(e^{tB/\eps^2}\zpw_t)^*\Psiz G(\Phi e^{tB/\eps^2}(\zpw_t+\mfw_t)) \\
&=\,\,(e^{tB/\eps^2}\zpw_t)^*\Psiz G(\Phi e^{tB/\eps^2}\zpw_t)\,\,+\,\,(e^{tB/\eps^2}\zpw_t)^*\Psiz \left(G(\Phi e^{tB/\eps^2}(\zpw_t+\mfw_t))-G(\Phi e^{tB/\eps^2}\zpw_t)\right).
\end{align*}
Using $G=G_1+G_3$, and the Lipschitz condition on the linear part $|G_1(\eta_1)-G_1(\eta_2)|\leq K_G||\eta_1-\eta_2||$, and that $||\Phi e^{tB/\eps^2}\mfw_t||\leq ||\mfw_t||_1$, and $|(e^{tB/\eps^2}\zpw_t)^*\Psiz|\leq ||\Psiz||_2||\zpw_t||_2$, we have
\begin{align}\label{eq:rateofevolnormzpw2}
\frac12\frac{d}{dt}||\zpw_t||_2^2\,\,\,&\leq\,\,\,K_G||\Psiz||_2||\zpw_t||_2^2+ (e^{tB/\eps^2}\zpw_t)^*\Psiz G_3(\Phi e^{tB/\eps^2}\zpw_t) \\ \notag
& \qquad \qquad +\,\,||\Psiz||_2||\zpw_t||_2K_G||\mfw_t||_1\,\,+\,\,\sum_{j=1}^3c_j||\zpw_t||_2^{4-j}||\mfw_t||_1^j,
\end{align}
for some constants $c_j>0$.

Define the time averaging operator $\Tavg$ as follows: For a periodic function $f:\R\to \R$ with period $2\pi/\om_c$, the action of $\Tavg$ is given by $\Tavg(f)=\frac{1}{2\pi/\om_c}\int_0^{2\pi/\om_c}f(s)ds$. Note that the condition \eqref{lem:add:propertyPTq_stable_condition} means that  
\begin{align}\label{eq:TavgG3conditionMeaning}
\Tavg \left((e^{\cdot B}z)^*\Psiz G_3(\Phi e^{\cdot B}z)\right) \,<\,-C_G||z||_2^4.
\end{align}
Define
\begin{align}\label{eq:def:G3tildefeq}
\tilde{G}_3(z,t):=(e^{t B}z)^*\Psiz G_3(\Phi e^{t B}z)-\Tavg ((e^{\cdot B}z)^*\Psiz G_3(\Phi e^{\cdot B}z)).
\end{align}
Then, using \eqref{eq:def:G3tildefeq} and \eqref{eq:TavgG3conditionMeaning} in \eqref{eq:rateofevolnormzpw2} we have
\begin{align*}
\frac12\frac{d}{dt}||\zpw_t||_2^2\,\,\,&\leq\,\,K_G||\Psiz||_2||\zpw_t||_2^2 \,\,-\,\,C_G||\zpw||_2^4 \,\,+\,\, \tilde{G}_3(\zpw_t,t/\eps^2) \\
& \qquad \qquad +\,\,||\Psiz||_2||\zpw_t||_2K_G||\mfw_t||_1\,\,+\,\,\sum_{j=1}^3c_j||\zpw_t||_2^{4-j}||\mfw_t||_1^j.
\end{align*}
Using Young's inequality we have for some $C_Y>0$
\begin{align}\label{eq:afterYoungEq}
\frac12\frac{d}{dt}||\zpw||_2^2\,\,&<\,\,-\frac12C_G||\zpw||_2^4 \,\,+\,\,C_Y||\mfw_t||_1^4\,\,+C_Y \,\,+\,\,\tilde{G}_3(\zpw_t,t/\eps^2).
\end{align}

Assume 
\begin{align}\label{eq:ass:supmfwLER}
\sup_{t\in[0,T]}||\mfw_t||_1^4<R,
\end{align}
and let $\widetilde{C}=C_Y(1+R)$. Then
\begin{align}\label{eq:lem:add:propertyPTq_stable_aux}
\frac12\frac{d}{dt}||\zpw||_2^2\,\,&<\,\,-\frac12C_G||\zpw||_2^4 \,\,+\,\,\widetilde{C}\,\,+\,\,\tilde{G}_3(\zpw_t,t/\eps^2).
\end{align}
Using comparison principle (theorem 6.1 on page 31 of \cite{HaleODEbook}), we have that $||\zpw_t||_2^2\leq \mfv_t$ where $\mfv$ is governed by
\begin{align}\label{eq:lem:add:propertyPTq_stable_aux_vsup}
\frac{d}{dt}\mfv_t\,\,&=\,\,-C_G\mfv_t+C_G(\mfv_t-\mfv_t^2) \,\,+\,\,2\widetilde{C}\,\,+\,\,2\tilde{G}_3(\zpw_t,t/\eps^2), \qquad \mfv_0=||\zpw_0||_2^2.
\end{align}
Using variation-of-constants formula, the fact that $\mfv-\mfv^2<1$, and integration-by-parts, we find that
\begin{align}\label{eq:lem:add:propertyPTq_stable_aux_vsup_upbound}
\mfv_t\,\,&<\,\,\mfv_0e^{-C_Gt}+\frac{2\widetilde{C}+C_G}{C_G}(1-e^{-C_Gt})+2\int_0^t\tilde{G}_3(\zpw_s,s/\eps^2)ds \\
& \qquad \qquad -C_G\int_0^te^{-C_G(t-s)}\left(2\int_0^s\tilde{G}_3(\zpw_u,u/\eps^2)du\right)ds. \notag
\end{align}
Now we try to obtain some bounds on the last two terms of the above inequality.

Using the structure of ${G}_3$ (defined in \eqref{eq:considerMAINadd}) and $\tilde{G}_3$ (defined in \eqref{eq:def:G3tildefeq}) and that $\Tavg(\tilde{G}_3(z,\cdot))=0$, it is easy to see that $\tilde{G}_3$ can be expressed as 
$$\tilde{G}_3(\zpw,t)=\sum_{j=1}^4(\alpha_j\cos(j\om_ct)+\beta_j\sin(j\om_ct))$$
where $\alpha_j$ and $\beta_j$ are fourth order polynomials in the components of $\zpw$. Define $$\remG(z,t):=2\int_0^t\tilde{G}_3(z,s)ds.$$
Using the structure of $\tilde{G}_3$ it is easy to see that (note that $\tilde{G}_3$ is mean zero and periodic as a function of its second argument) there exists $C_{\remG}>0$ such that
$$|\remG(\zpw,t)|\leq C_{\remG}(1+||\zpw||_2^4), \qquad \qquad ||\frac{\partial \remG}{\partial \zpw}(\zpw,t)||_2\leq C_{\remG}(1+||\zpw||_2^3).$$
Also, from \eqref{eq:howzpwevolves}, it is easy to see that $\exists C_*>0$ such that $||\dot{\zpw}||_2\leq C_*(1+||\zpw+\mfw||_2^3)$.
Since
\begin{align*}
\eps^2\remG(\zpw_t,t/\eps^2)-\eps^2\remG(\zpw_0,0)-\eps^2\int_0^t\frac{\partial \remG}{\partial \zpw}(\zpw_s,s/\eps^2)\dot{\zpw}_sds = 2\int_0^t\tilde{G}_3(\zpw_s,s/\eps^2)ds,
\end{align*}
and $\remG(\zpw,0)=0$, we have
\begin{align*}
 \left|2\int_0^t\tilde{G}_3(\zpw_s,s/\eps^2)ds\right| \,\,&\leq\,\,\eps^2C_{\remG}(1+||\zpw_t||_2^4)+\eps^2C_{\remG}C_*\int_0^t(1+||\zpw_s||_2^3)(1+||\zpw_s+\mfw_s||_2^3)ds \\ 
 &\leq\,\,\eps^2C_{\remG}(1+||\zpw_t||_2^4)+\eps^211C_{\remG}C_*\int_0^t(1+||\zpw_s||_2^6+||\mfw_s||_1^6)ds.
\end{align*}
Let 
\begin{align}\label{eq:auxeqtauepsdef_lem:add:propertyPTq_stable}
\tau^\eps := \inf \{t\geq 0\,:\,||\zpw_t||_2 \geq \frac{1}{\eps^{1/6}}\}.
\end{align}
Then, for $t \leq \min\{\tau^\eps, T\}$ we have
\begin{align*}
\left|2\int_0^t\tilde{G}_3(\zpw_s,s/\eps^2)ds\right| \,\,&\leq \,\,\eps^2C_{\remG}+\eps^{4/3}C_{\remG}  + \eps^211C_{\remG}C_*\int_0^t(1+\eps^{-1}+R^{3/2})ds. 
\end{align*}
 When $\eps<1$, we have (from the above inequality) that  for $t \leq \tau^\eps \wedge T$
\begin{align}
\left|2\int_0^t\tilde{G}_3(\zpw_s,s/\eps^2)ds\right| \,\,&\leq\,\,\eps^{4/3}2C_{\remG}\,\,+\,\,\eps\widehat{C}t
\label{eq:lem:add:propertyPTq_stable_aux_intg_aftbou_simp}
\end{align}
where $\widehat{C}=22C_{\remG}C_*(1+R^{3/2})$.

Using \eqref{eq:lem:add:propertyPTq_stable_aux_intg_aftbou_simp} in \eqref{eq:lem:add:propertyPTq_stable_aux_vsup_upbound}, we have\footnote{For the last term in RHS of \eqref{eq:lem:add:propertyPTq_stable_aux_vsup_upbound} we have used that $$|C_G\int_0^te^{-C_G(t-s)}f(s)ds| \quad \leq \quad (\sup_{s\in [0,t]}|f(s)|)\,C_G\int_0^te^{-C_G(t-s)}ds\quad \leq\quad (\sup_{s\in [0,t]}|f(s)|).$$} for $\eps<1$ and $t \leq \tau^\eps \wedge T$
\begin{align}\label{eq:lem:add:propertyPTq_stable_aux_vsup_upbound_up}
||\zpw_t||_2^2\,\,\leq\,\,\mfv_t\,\,&<\,\,\mfv_0e^{-C_Gt}+\frac{2\widetilde{C}+C_G}{C_G}(1-e^{-C_Gt})+\eps2\widehat{C}t+\eps^{4/3}4C_{\remG} \\ \notag
&<\,\,\max\{||\zpw_0||_2^2,\,\frac{2\widetilde{C}+C_G}{C_G}\}+\eps2(\widehat{C}T+2C_{\remG}).
\end{align}
Hence, for $\eps<1$ and $t \leq \tau^\eps \wedge T$
\begin{align*}
||\zpw_t||_2\,\,<\,\,||\zpw_0||_2+1+\sqrt{\frac{2\widetilde{C}}{C_G}}+\sqrt{\eps}\sqrt{2\widehat{C}T}+\sqrt{\eps}\sqrt{4C_{\remG}}. 
\end{align*}
Using $\widehat{C}=22C_{\remG}C_*(1+R^{3/2})$ and that $\widetilde{C}=C_Y(1+R)$ we find that
\begin{align*}
||\zpw_t||_2\,\,&<\,\,||\zpw_0||_2+(1+\sqrt{\eps}\sqrt{4C_{\remG}})+\sqrt{\frac{2C_Y}{C_G}}\sqrt{1+R}+\sqrt{\eps}\sqrt{44C_{\remG}C_*T}(1+R^{3/4}). 
\end{align*}
Note that  $\exists \,\eps_{(2)}$ such that $\forall \eps < \eps_{(2)}$ we have $\sqrt{\eps}\sqrt{4C_{\remG}}<1$. Also,  $\exists \,\eps_{(3)}$ such that $\forall \eps < \eps_{(3)}$ we have $\sqrt{\eps}\sqrt{\frac{44C_{\remG}C_*T}{2C_Y/C_G}}<1$.
Hence, for $\eps<\min\{1,\eps_{(2)},\eps_{(3)}\}=:\eps_{(4)}$ and $t \leq \tau^\eps \wedge T$ we have\footnote{We use $\sqrt{1+R}+(1+R^{3/4})\,\,\, <\,\,\, 6\sqrt{1+R^{6/4}}$.}
\begin{align}\label{eq:lem:add:propertyPTq_stable_aux_zpw_upbound_upup}
||\zpw_t||_2\,\,<\,\,||\zpw_0||_2+2+6\sqrt{\frac{2C_Y}{C_G}}\sqrt{1+R^{6/4}}.
\end{align}
Hence, for $\eps<\eps_{(4)}$ if $\tau^\eps \geq T$ we have (using $||\zedh_t||_2\leq ||\zpw_t||_2+||\mfw_t||_1$)
 \begin{align*}\sup_{t\in[0,T]}||\zedh_t||_2\,\,&<\,\, ||\zedh_0||_2+2+6\sqrt{\frac{2C_Y}{C_G}}\sqrt{1+R^{6/4}} + R^{1/4}\\ &<\,\,||\zedh_0||_2+2+6\left(\frac{1}{6}+\sqrt{\frac{2C_Y}{C_G}}\right)\sqrt{1+R^{6/4}} \\
 & =: \quad ||\zedh_0||_2+2+C_{YG}\sqrt{1+R^{6/4}} 
 \end{align*}
Because the initial condition is deterministic $\exists C_0>0$ such that $||\zedh_0||_2<C_0$. Hence
 \begin{align}\label{eq:proof:aux:howmuchboundonz}
 \sup_{t\in[0,T]}||\zedh_t||_2\,\,&<\,\, C_0 +2+C_{YG}\sqrt{1+R^{6/4}}. 
 \end{align}

Define $C_R$ by
$$C_R  \overset{{\tt def}}{=} C_0 + 2+C_{YG}\sqrt{1+R^{6/4}}.$$
For $\eps<\eps_{(4)}$ and $t \leq \tau^\eps \wedge T$, we have  from \eqref{eq:lem:add:propertyPTq_stable_aux_zpw_upbound_upup} that $||\zpw_t||_2 < C_R$. So, if we define $\eps_R:=(1/C_{R})^6$, then for $\eps < \min\{ \eps_{(4)},\eps_R\}$ we have that $||\zpw_t||_2 < \frac{1}{\eps^{1/6}}$ and hence $\tau^\eps >T$. Hence, from \eqref{eq:proof:aux:howmuchboundonz} we have that for $\eps < \min\{ \eps_{(4)},\eps_R\}$
 \begin{align}\label{eq:proof:aux:howmuchboundonz_nice}
 \sup_{t\in[0,T]}||\zedh_t||_2\,\,&<\,\, C_R. 
 \end{align}
Recalling the definition of $R$ from \eqref{eq:ass:supmfwLER} we have for $\eps < \min\{ \eps_{(4)},\eps_R\}$
\begin{align}\label{eq:proof:aux:auxauxnice}
\mbbP\left[\sup_{t\in[0,T]}||\zedh_t||_2\,\geq\,  C_R\right]\,\,&\leq\,\,\mbbP\left[\sup_{s\in[0,T]}||\mfw_t||_1\,\geq \, \left( \left( \frac{C_R-C_0-2}{C_{YG}}  \right)^2-1 \right)^{1/6} \right].
\end{align}
Lets estimate the RHS of the above equation. Using the definition  of $\mfw$ and then Markov and Burkholder-Davis-Gundy inequalities we have that 
\begin{align*}
\mbbP\left[\sup_{s\in[0,T]}||\mfw_t||_1\,\geq \, \rho \right] \,\,\,&\leq\,\,\,\sum_{j=1}^2\mbbP\left[\sup_{s\in[0,T]}\left|\int_0^t(e^{-sB/\eps^2}\Psiz)_j \sigma dW_s\right|\,\geq \, \frac12\rho \right] \\
&\leq\,\,\,\frac{2}{\rho}\sum_{j=1}^2\expt\left[\sup_{s\in[0,T]}\left|\int_0^t(e^{-sB/\eps^2}\Psiz)_j \sigma dW_s\right|\right] \\
&\leq\,\,\,\frac{2C_{bdg}}{\rho}\sum_{j=1}^2\expt \sqrt{\int_0^t(e^{-sB/\eps^2}\Psiz)_j^2 \sigma^2 ds}  \quad \leq \quad \frac{2\sqrt{2}C_{bdg}}{\rho}||\Psiz||_2\sigma \sqrt{T} 
\end{align*}
Hence, from \eqref{eq:proof:aux:auxauxnice} we have that for $\eps < \min\{ \eps_{(4)},\eps_R\}$
\begin{align}\label{eq:proof:aux:auxauxnice2}
\mbbP\left[\sup_{t\in[0,T]}||\zedh_t||_2\,\geq\,  C_R\right]\quad &\leq\quad \frac{2\sqrt{2}C_{bdg}||\Psiz||_2\sigma \sqrt{T}}{\left( \left( \frac{C_R-C_0-2}{C_{YG}}  \right)^2-1 \right)^{1/6}} \quad \overset{{\tt def}}{=}: \quad f(C_R).
\end{align}

Let ${C}_{R,q}>C_0+2+C_{YG}$ be such that $f(C_R)<q$ for $C_R>C_{R,q}$.
Such a ${C}_{R,q}$ exists because $f$ is monotonically decreasing in $C_R$ (for $C_R>C_0+2+C_{YG}$). Choose $C_{\stopt}>C_{R,q}/0.99$ and $\eps_*=\min\{\eps_{(4)},(C_{R,q})^{-6}\}$. Let $E^\eps$ be as defined in \eqref{assonzedh_setEeps}. Now, using \eqref{eq:proof:aux:auxauxnice2} for $\eps < \eps_*$
\begin{align*}
\mbbP[\Om \setminus E^\eps]\,\,& =  \mbbP\left[\sup_{t\in [0,T]}||\Phi e^{tB/\eps^2}\zedh_t||\,\geq\,0.99C_{\stopt}\right]\\
&\leq  \mbbP\left[\sup_{t\in [0,T]}||\zedh_t||_2\,\geq\,0.99C_{\stopt}\right]\,\, \leq\,\,\mbbP\left[\sup_{t\in[0,T]}||\zedh_t||_2\geq C_{R,q}\right]\,\,\leq \,\,f(C_{R,q})\,\,<\,\,q.
\end{align*}
Hence $\mbbP[E^\eps]\,\,\geq\,\,1-q,$ and so \eqref{eq:considerMAINadd} possesses the property $\mathscr{P}(T,q)$. As mentioned above, for any $q>0$ it is possible to select $\eps_*$ and $C_{R,q}$ such that $f(C_{R,q})<q$ and hence  $\zedh$  possesses the property $\mathscr{P}(T,q)$ for arbitrary $q>0$.

\subsection{Proof of proposition \ref{prop:betaepsSmallAddNoiseFin}}\label{sec:proofofprop_betaepsSmallAddNoiseFin}

Using Ito formula, $\beta^\eps_t$ satisfies $$d\beta^\eps_t=\B_tdt, \qquad \beta^\eps_0=0,$$ where 
$$\B_t=\Gamma_t\left(G(\Phi e^{tB/\eps^2}\zedt_t)-G(\Phi e^{tB/\eps^2}\zedh_t+\ydec_t)\right), \qquad \Gamma_t=\sum_{i=1}^{2}\left(e^{-tB/\eps^2}\Psiz\right)_i\left(\zedt_t-\zedh_t\right)_i.$$
Using the structure of $e^{-tB}$ and $\Psiz$, we have $|\Gamma_t|\leq \sqrt{\Psiz^*\Psiz}\sqrt{2\beta^\eps_t}$. Writing $\zedt_t$ as $\zedh_t+(\zedt_t-\zedh_t)$ and expanding $G$ in $\B_t$ we get
\begin{align*}
|\B_t|\,\leq\,C\sqrt{\beta^\eps_t}\left(\sqrt{\beta^\eps_t}+||\ydec^\eps_t|| + \sum_{j=1}^{3}||\Phi e^{tB/\eps^2}\zedh_t||^{3-j}\left((\sqrt{\beta^\eps_t})^j+||\ydec^\eps_t||^j\right)\right)
\end{align*}
Because the nonlinearity is such that \eqref{lem:add:propertyPTq_stable_condition} is satisfied, by lemma \ref{lem:add:propertyPTq_stable}, $\zedh$ possesses property $\mathscr{P}(T,q)$ for abitrary $q>0$. Hence, it is possible to select $C_{\stopt}>0$  so that $\exists\,\eps_*>0$ such that $\forall \eps<\eps_*$, we have $\mbbP[E^\eps]\geq1-q$ where $E^\eps$ is defined in \eqref{assonzedh_setEeps}. So, with probability at least $1-q$ we have
\begin{align*}
|\B_t|\,\leq\,C\sqrt{\beta^\eps_t}\left((1+C_{\stopt}^2)\sqrt{\beta^\eps_t} + C_{\stopt}(\sqrt{\beta^\eps_t})^2+(\sqrt{\beta^\eps_t})^3 +(1+C_{\stopt}^2)\sum_{j=1}^3||\ydec^\eps_t||^j \right).
\end{align*}
Let $\mathscr{C}:=||(1-\pi)\hpex{0}||$.
As long as $\sqrt{\beta^\eps_t}<1$ we have
\begin{align*}
|\B_t|\,\leq\,C\sqrt{\beta^\eps_t}\left((3+2C_{\stopt}^2)\sqrt{\beta^\eps_t}+(1+C_{\stopt}^2)\left(\sum_{j=1}^3\mathscr{C}^j\right)e^{-kt/\eps^2} \right).
\end{align*}
Hence, as long as $\sqrt{\beta^\eps_t}<1$
\begin{align*}
d\sqrt{\beta^\eps_t}\,\leq\,C(3+2C_{\stopt}^2)\sqrt{\beta^\eps_t}+C(1+C_{\stopt}^2)\left(\sum_{j=1}^3\mathscr{C}^j\right)e^{-kt/\eps^2}.
\end{align*}
Using Gronwall, we get (as long as $\sqrt{\beta^\eps_t}<1$)
\begin{align*}
\sqrt{\beta^\eps_t} \,&\leq\, \frac{\eps^2 C(1+C_{\stopt}^2)(\sum_{j=1}^3 \mathscr{C}^j)}{k+\eps^2C(3+2C_{\stopt}^2)} \left(e^{C(3+2C_{\stopt}^2)t}-e^{-kt/\eps^2}\right) \\
& < \frac{\eps^2 C(1+C_{\stopt}^2)(\sum_{j=1}^3 \mathscr{C}^j)}{k} e^{C(3+2C_{\stopt}^2)T}
\end{align*}
Choose $\eps_{**}$ small enough so that the above expression is less than one. Set $\eps_{\circ}=\min\{\eps_{*},\eps_{**}\}$. Then, we have $\forall \eps <\eps_{\circ}$, $\sup_{t\in [0,T]}\beta^\eps_t < C\eps^4$ with probability at least $1-q$. $\hfill \qed$

\end{document}